\documentclass[11pt,oneside,notitlepage]{article}
\usepackage{geometry}

\usepackage{lmodern}

\usepackage[ngerman,english]{babel}
\usepackage{amsmath}
\usepackage{amsrefs}
\usepackage{amssymb}
\usepackage{amsthm}
\usepackage{caption,subcaption}
\usepackage{cite}
\usepackage{color}
\usepackage{enumitem}
\usepackage{framed}
\usepackage{graphicx}
\usepackage{hyperref}
\usepackage{layout}
\usepackage{mathtools}
\usepackage{mathrsfs}
\usepackage{placeins}
\usepackage{pdfpages}

\usepackage{tikz}
\usepackage{txfonts}
\usepackage[normalem]{ulem}

\newtheorem{theorem}{Theorem}

\newtheorem{corollary}[theorem]{Corollary}

\newtheorem{lemma}[theorem]{Lemma}
\newtheorem{notation}[theorem]{Notation}
\newtheorem{proposition}[theorem]{Proposition}
\newtheorem{remark}[theorem]{Remark}

\newcommand{\sgn}{{\rm sgn}}

\newcommand{\eps}{\varepsilon}

\newcommand{\eb}{\bar{\mathcal{E}}}
\newcommand{\f}{\bar{f}}
\newcommand{\ff}{f}
\newcommand{\F}{\bar{F}}
\newcommand{\operatorL}{\ensuremath{\mathcal{L}}}

\DeclarePairedDelimiter\abs{\lvert}{\rvert}
\DeclarePairedDelimiter\norm{\lVert}{\rVert}
\newcommand{\dd}{\ensuremath{\mathrm{d}}}
\newcommand{\energy}{\ensuremath{E}}
\newcommand{\step}{\underline}
\newcommand{\leqsim}{\ensuremath{\lesssim}}
\newcommand{\geqsim}{\ensuremath{\gtrsim}}
\newcommand{\dissipation}{\ensuremath{D}}
\newcommand{\torus}{\mathbb{T}}
\newcommand{\NN}{N_{1}}

\newcommand{\depEH}{\scriptscriptstyle{C_{H},C_{E}}}
\newcommand{\depE}{\scriptscriptstyle{C_{E}}}

\newcommand{\sppt}{\text{sppt}}
\newcommand{\R}{\mathbb{R}}
\newcommand{\N}{\mathcal{N}}
\newcommand{\M}{\mathcal{M}}

\newcommand{\E}{\mathcal{E}}
\newcommand{\Ez}{\mathcal{E}_{0}}

\definecolor{darkred}{rgb}{0.9,0.1,0.1}

\def\Xint#1{\mathchoice
{\XXint\displaystyle\textstyle{#1}}%
{\XXint\textstyle\scriptstyle{#1}}%
{\XXint\scriptstyle\scriptscriptstyle{#1}}%
{\XXint\scriptscriptstyle\scriptscriptstyle{#1}}%
\!\int}
\def\XXint#1#2#3{{\setbox0=\hbox{$#1{#2#3}{\int}$}
\vcenter{\hbox{$#2#3$}}\kern-.5\wd0}}

\def\dashint{\Xint-}

\numberwithin{equation}{section}   
\numberwithin{theorem}{section}

\begin{document}

\title{Metastability of the Cahn--Hilliard equation\\ in one space dimension}
\author{Sebastian Scholtes and Maria G. Westdickenberg}
\newcommand{\Addresses}{{
  \bigskip
  \footnotesize

  	\noindent \textsc{Sebastian Scholtes, RWTH Aachen University}\\
  	\textit{E-mail address:} \texttt{scholtes@math1.rwth-aachen.de}

  	\medskip

	\noindent \textsc{Maria G. Westdickenberg, RWTH Aachen University}\\
  	\textit{E-mail address:} \texttt{maria@math1.rwth-aachen.de}

}}
\date{\today}
\maketitle
\begin{abstract}
We establish metastability of the one-dimensional Cahn--Hilliard equation for initial data that is order-one in energy and  order-one in $\dot{H}^{-1}$ away from a point on the so-called slow manifold with $N$ well-separated layers. Specifically, we show that, for such initial data on a system of lengthscale $\Lambda$, there are three phases of evolution: (1) the solution is
drawn after a time of order $\Lambda^2$ into an algebraically small neighborhood of the $N$-layer branch of the slow manifold, (2) the solution is drawn after a time of order $\Lambda^3$ into an exponentially small neighborhood of the $N$-layer branch of the slow manifold, (3) the solution is trapped for an exponentially long time exponentially close to the $N$-layer branch of the slow manifold. The timescale in phase (3) is obtained with the sharp constant in the exponential.
\end{abstract}

\section{Introduction}\label{S:intro}
Local energy minimizers are the stable states of a gradient flow:  Solutions started at the minimizers are in equilibrium, and solutions started nearby relax towards this equilibrium state.  There are however physical systems that exhibit \emph{metastability}: Solutions appear to be stationary, but are in fact far from any stable state and evolving on an extremely long timescale.  (This behavior is called dynamic metastability to distinguish it from the noise-induced metastability of stochastic systems.)

Two fundamental examples displaying dynamic metastability are the one-dimensional Allen--Cahn equation
\begin{align}
u_t=u_{xx}-G'(u),\qquad x\in \left(-\tfrac{\Lambda}{2},\tfrac{\Lambda}{2}\right),\;t>0,\label{ac}
\end{align}
and the one-dimensional Cahn--Hilliard equation
\begin{align}
u_t=-(u_{xx}-G'(u))_{xx},\qquad x\in \left(-\tfrac{\Lambda}{2},\tfrac{\Lambda}{2}\right),\;t>0,\label{ch}
\end{align}
subject to suitable boundary conditions. Here $G(u)$ is a double-well potential with nondegenerate minima at $\pm 1$ (cf. remark \ref{rem:G}).
Both~\eqref{ac} and~\eqref{ch} are often studied with a small parameter $\eps$ appearing in the equation.  This is equivalent via rescaling to studying~\eqref{ac} and~\eqref{ch} on an interval with $\Lambda\gg 1$, which is the setting that we will consider in this paper.

Both equations represent phenomenological models for the coexistence of two ``pure'' phases $\pm 1$, and the value of the order parameter $u$ indicates the proportion of each phase.
An important physical and mathematical property of the Cahn--Hilliard equation (under appropriate boundary conditions) is that it preserves the mean:
\begin{align*}
  \dashint_{\left(-\frac{\Lambda}{2},\frac{\Lambda}{2}\right)} u(x,t)\,\dd x=\dashint_{\left(-\frac{\Lambda}{2},\frac{\Lambda}{2}\right)}  u(x,0)\,\dd x=:m\qquad\text{for all }t>0.
\end{align*}

Equations~\eqref{ac} and~\eqref{ch} can be derived from the scalar Ginzburg--Landau energy
\begin{align}
E(u):=\int_{\left(-\frac{\Lambda}{2},\frac{\Lambda}{2}\right)}  \frac{1}{2}u_x^2 +G(u)\,\dd x.\label{energy}
\end{align}
Equations~\eqref{ac} and \eqref{ch} are the gradient flows of~\eqref{energy} with respect to the $L^2$ metric and the $\dot{H}^{-1}$ metric, respectively.

The metastability of equation~\eqref{ac} has been well-analyzed; see \cite{FH,CP,BK,C,OR} and the discussion in subsection \ref{ss:lit}. The generic picture can be described in the following way.
For initial data with large regions of positive phase interspersed with large regions of negative phase, it is observed that the solution quickly settles down to a configuration with large regions of $u\approx\pm 1$ that are connected by so-called transition layers. These transition layers are well-approximated by  energy minimizers on $\R$ connecting $\pm 1$ boundary conditions at infinity.  Subsequently the solution appears almost stationary until a time that is of exponential order with respect to the distance between zeros.  Roughly speaking, the collection of states with $N$ optimal transition layers connecting $\pm 1$ forms a slow motion manifold for the system: The system quickly relaxes to the slow manifold and then evolves slowly along it.  However the solution is far from
the final state.  Indeed, suppose that the two closest transition layers are a distance $\ell$ apart. After a time that is exponentially long with respect to $\ell$, these two layers come together and collapse, reducing the energy and producing a state with $N-2$ transitions.  Again the solution remains metastable for an exponentially long time until the next two layers come together, and so on, until the last pair of layers collapses.  Because of the excess energy in the system around the time of a collision, tracking the evolution through collisions requires controlling initial data that is order one away from the slow manifold.

The same basic picture of metastable evolution sketched above for the Allen--Cahn equation holds for the Cahn--Hilliard equation~\eqref{ch}.  It has been proved that solutions of the Cahn--Hilliard equation are metastable for initial data  \emph{sufficiently close to the slow
manifold} \cite{BH,BX1,BX2,G}.
More is true, however:
Numerical studies show that, as for the Allen--Cahn equation, there
is a large class of initial data \emph{order one away} from the slow manifold that  is \emph{drawn into} a small neighborhood of the slow manifold during an initial relaxation stage; see \cite{EF,SW,num} and also figure \ref{fig:min}.  Theorem \ref{t:main1} below establishes this result. Specifically, we show that order one initial data is drawn within time $O(\Lambda^2)$ into an algebraically small neighborhood of the slow manifold and within time $O(\Lambda^3)$ into an exponentially small neighborhood of the slow manifold. Thereafter, the solution is trapped in the exponentially small neighborhood of the slow manifold for an exponentially long time.
\begin{figure}[t!]
    \centering
    \begin{subfigure}[t]{0.5\textwidth}
        \centering
        \includegraphics[height=2in]{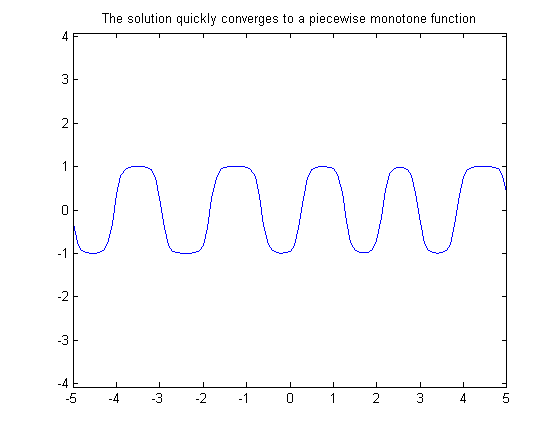}
        \caption{Transition layer structure of the solution.}
    \end{subfigure}%
    ~
    \begin{subfigure}[t]{0.5\textwidth}
        \centering
        \includegraphics[height=2in]{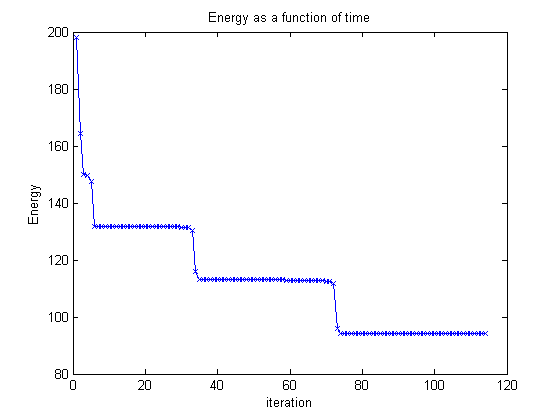}
        \caption{Long energy plateaus.}
    \end{subfigure}
    \caption{The solution of the Cahn-Hilliard equation quickly relaxes to a collection of nearly optimal transition layers.}\label{fig:min}
\end{figure}

A motivation for \cite{OR} was to understand  the features of an energy landscape that lead to metastable dynamics of a gradient flow. By analogy with the finite dimensional case, one expects that the energy landscape should have some sort of convexity transverse to the slow manifold and be fairly flat along it.  These properties are codified in the sufficient conditions for metastability set forth in~\cite{OR} (see also proposition \ref{prop:orweak2}, below), which were in the same paper applied to give a new proof of metastability of the Allen--Cahn equation with the sharp exponential timescale. An energy--energy--dissipation inequality provides the convexity and a Lipschitz condition captures the level grade of the slow manifold.

The fourth order Cahn--Hilliard equation~\eqref{ch} is more complicated than~\eqref{ac} in part because of the absence of maximum principles.  As a result, excess energy can lead to the formation and disappearance of ``spurious'' zeros in the neighborhood of a transition layer, complicating the analysis. Even more important is the absence of a spectral gap for the  problem on $\R$. Although we will see that an energy--energy--dissipation inequality holds on compact intervals (cf. lemma \ref{l:eed}, below), the constant in this inequality depends on the system size and diverges as $\Lambda\to\infty$. Because of this (unavoidable) system-size dependence, a direct application of the buckling argument from \cite{OR} fails to control order-one initial data. Put differently, on the timescale that the method from \cite{OR} would need to show that the energy gap had become small, initial transition layers could already have moved a significant distance.

To overcome this difficulty, we use the so-called relaxation framework introduced in \cite{OW} to control the
initial phase of the evolution, in which the energy gap relaxes from order one to algebraically small as an \emph{algebraic but system-size independent function of time}. Once the energy gap has become sufficiently small, we apply the metastability framework from \cite{OR} in the form of proposition \ref{prop:orweak2}. In terms of the preceding metaphorical description of phase space, this approach can be described as establishing and making use of non-strict convexity away from the slow manifold and strict but system-size dependent convexity near the slow manifold.
Both the relaxation framework and the metastability framework may be useful for other applications including other higher order equations and systems.

\subsection{Setting and main result}\label{ss:setting}
We consider the solution $u$ of the Cahn--Hilliard equation \eqref{ch}
subject to periodic boundary conditions on $\torus_{\Lambda}:=\left(-\frac{\Lambda}{2},\frac{\Lambda}{2}\right]$. One could, with minor modifications, consider the case of Neumann or Dirichlet boundary conditions. To fix ideas, we consider  mean value zero; one could equally well consider any fixed $m\in(-1,1)$.

\begin{remark}\label{rem:G}
 The canonical potential is
\begin{align*}
G(u)=\frac{1}{4}\left(1-u^2\right)^2,
\end{align*}
but $G$ can be any nondegenerate double-well potential; explicitly, we assume:
\begin{itemize}
	\item
		$G$ is $C^2$ and even,
	\item
		$G(u)>0$ for $u\not=\pm 1$ and $G(\pm 1)=0$,
	\item
		$G'(u)\leq 0$ for $u\in [0,1]$ and $G''(\pm 1)>0$.
\end{itemize}
The assumption that $G$ is even could be relaxed.
\end{remark}


We will now explain the set that we will regard as the slow manifold of the evolution.
The collection of so-called alternating-sign energy-optimal profiles with an even number $N$ of well-separated zeros is the set:
\begin{align*}
  \N_N(\ell):&=\Big\{v\in  H^1\left(\torus_{\Lambda}\right) \text{ periodic with $N$ simple zeros }x_1,\ldots, x_N, \text{ such that $|x_{i+1}-x_i|\geq \ell,$}\\
  & v \text{ minimizes the energy on $(x_i,x_{i+1})$},\text{ and $v$ changes sign on consecutive intervals} \Big\}.
\end{align*}
Throughout, distances are interpreted in the periodic sense; i.e., $x_{N+1}:=\Lambda+x_1$.
A function $v\in\N_{N}(\ell)$ is in $C\left(\torus_{\Lambda}\right)\cap H^1\left(\torus_{\Lambda}\right)$ and satisfies
\begin{align}\label{stationary_AC}
	-v_{xx}+G'(v)=0\qquad\text{on}\quad (x_{i},x_{i+1}).
\end{align}
The set $\mathcal{N}_N(\ell)$ comprises the $N$ layer branch of our slow manifold.

The metastable set with which we will work is the set of functions with order one energy and order one $\dot{H}^{-1}$-distance  to some member of the slow manifold:
\begin{align*}
  \mathcal{M}_N(\ell,C_H,C_E):&=\Big\{ u\in H^1\left(\torus_{\Lambda}\right) \text{ periodic with }\dashint u\,\dd x=0, \text{ such that } E(u)\leq C_E,\\
   &\quad\text{ and there exists }\bar{v}\in\N_N(\ell)\text{ such that }\norm{u-\bar{v}}_{\dot{H}^{-1}}^2\leq C_H   \Big\}.
\end{align*}
Above and throughout, integrals are over $\torus_\Lambda$ unless otherwise indicated.

We remark that the finite $\dot{H}^{-1}$ norm of $u-\bar{v}$ implies that $\dashint \bar{v}\,\dd x=\dashint u\,\dd x$ (which we have assumed to be zero).
Note that $u$ and $\bar{v}$ do not have the same zeros, and the $\ell$ in the definition of the set does not refer to the distance between zeros of $u$, but rather to the distance between zeros of $\bar{v}$. It will be important throughout the following, however, that $u$ does have $N$ well-separated zeros and is order one close in $L^2$ to an alternating-sign energy-optimal profile with precisely these zeros; this is the content of the following lemma.
\begin{lemma}\label{l:uv}
	For any $C_H,\,C_E\in (0,\infty)$, there exists $\ell_1\in (0,\infty)$ with the following property.
	For any $\ell\geq\ell_1$ and $u\in\mathcal{M}_N(\ell,C_H,C_E)$, let $\bar{c}$ denote the zeros of an associated $\bar{v}\in \mathcal{N}_N(\ell)$ and $\bar{H}:=\norm{u-\bar{v}}_{\dot H^{-1}}^{2}$.
	Then $u$ has $N$ zeros $c=\{x_1,\ldots,x_N\}$ such that
	\begin{align}\label{ddistancezeroes}
		\abs{c-\bar{c}}\lesssim \bar{H}^\frac{1}{3}+\bar{H}^{\frac{1}{5}}\qquad\text{and}\qquad
		N\lesssim E(u).
	\end{align}
	Moreover, there exists $v\in\mathcal{N}_N(\frac{\ell}{2})$ whose zeros are given by $c$ and such that
	\begin{align}\label{distance_uvc}
		\norm{u-v}_{L^2}^2\lesssim \bar{H}^{\frac{1}{3}}+\bar{H}^{\frac{1}{5}}+\bar{H}^{\frac{1}{2}}\left(\bigl(E(u)\bigr)^{\frac{1}{2}}+1\right).
	\end{align}
\end{lemma}
We prove lemma \ref{l:uv} in subsection \ref{ss:prelim}. For an explanation of the $\lesssim$ notation, see notation \ref{notation:sim}, below.
\begin{notation}\label{not:cdist}
Above we have introduced the  notation $|c-\bar{c}|$ for $c=(x_1,\ldots,x_N),\, \bar{c}=(\bar{x}_1,\ldots,\bar{x}_N)$ to denote the maximum distance between ``corresponding zeros,'' by which we mean
  \begin{align*}
 |c-\bar{c}|:=   \min\left\{ \max_{i\in 1,\ldots,N}|x_i-\bar{x}_i|\,,\, \max_{i\in 1,\ldots, N} |x_{i+1}-\bar{x}_i|\,,\, \max_{i\in 1,\ldots, N} |x_{i}-\bar{x}_{i+1}| \right\}.
  \end{align*}
  By shifting in $x$, we may  for notational simplicity assume $|c-\bar{c}|=\max_{i\in 1,\ldots,N}|x_i-\bar{x}_i|.$
\end{notation}
\begin{remark}
  Here the existence of $N$ zeros with the given property does not rule out---and is not disturbed by---the existence of additional zeros of $u$. We do not assume that the zeros $c$ associated by lemma \ref{l:uv} to a solution are uniquely determined or continuous in time. However we will show that $u$ has precisely $N$ zeros (and hence, that $c$ and $v$ are uniquely defined and continuous in time) after an initial relaxation stage; cf. theorem \ref{t:main1}.
\end{remark}

\uline{Throughout the paper} we will for  $u\in\M_N(\ell,C_H,C_E)$ let $c$ and $v$ denote the zeros and a corresponding energy-optimal profile, respectively, as in lemma \ref{l:uv}. Whereas $\bar{v}$ is constant in time, $c$ and $v$ depend on time.
We introduce the $v$-related quantities
\begin{align*}
  \E(t):=E\bigl(u\bigr)-E\bigl(v\bigr)\qquad\text{and}\qquad f(t):=u-v
\end{align*}
as well as the $\bar{v}$-related quantities
\begin{align*}
  \eb(t):=E\bigl(u\bigr)-E(\bar{v}),\qquad\f(t):=u-\bar{v},\qquad\text{and}\qquad  \bar{H}(t):=\norm{u-\bar{v}}_{\dot{H}^{-1}}^2.
\end{align*}
We define in addition the dissipation
\begin{align*}
  D(t):=\int \biggl(\Bigl(-u_{xx}+G'\bigl(u\bigr)\Bigr)_x\biggr)^2\,\dd x
\end{align*}
as the negative of the time-derivative of the energy of $u$.
Finally, we define the constant
\begin{align}
  C_G:=\sqrt{G''(\pm 1)},\label{cg}
\end{align}
which gives the sharp exponential timescale of metastability (see \eqref{expt}).
Our main result is:
\begin{theorem}\label{t:main1}
	We fix the number of zeros $N$ and will not track $N$-dependence below.
  Consider $C_H,\,C_E,\,C_1\in (1,\infty)$. Let $C_{\mathrm{ed}}$ be the constant from lemma \ref{l:eed}. There exists $\Lambda_1\in(0,\infty)$ such that for any $\Lambda\geq \Lambda_1$ and
  \begin{align}
    \ell_0\geq \frac{\Lambda}{C_1},\quad \bar{H}_0\leq C_H,\quad E_0\leq C_E\label{ellbig}
  \end{align}
  the following holds true.
 The solution $u(t)$ of the Cahn--Hilliard equation \eqref{ch} with initial data $u_0\in \M_N(\ell_0, \bar{H}_0,E_0)$ satisfies the following. For
 \begin{align}
   \delta :=\Lambda^{-\frac{1}{2}}\exp(-C_G\ell(0))
   \quad\text{ and for all }\quad t\leq \delta^{-1},\label{expt}
 \end{align}
the solution is in $\M_N(\ell_0, \bar{H}(t),E_0)$ and $\bar{H}(t)\lesssim 1$. Moreover the following phases of the evolution exist.
Below  $c(t)$ and $\bar{c}$ denote the zeros of $u$ and $\bar{v}$ as in lemma \ref{l:uv}, $\ell(0)$ is the minimal distance between the zeros in $c(0)$, $v\in\N_N(\frac{\ell_0}{2})$ is any energy-optimal profile associated to $u$, and $\Ez$ denotes the initial energy gap.
%
\begin{enumerate}
  \item[(i)] There exists $s_1\lesssim  (\bar{H}_{0}+\Ez+\Ez^{3})\Lambda^2$ such that on $[0,s_1]$ there holds
\begin{align}
 \norm{u(t)-v(t)}_{H^1}^2+ \E(t)&\lesssim \min\Big\{\Ez,\,\frac{\bar{H}_0+\Ez+\Ez^3}{t} \Big\},\label{egap1}\\
  (c(t)-\bar{c})^2&\lesssim \left(\bar{H}_0+\Ez+\Ez^3\right)^\frac{1}{2}\E(t)^\frac{1}{2},\notag\\
  \bar{H}(t)&\lesssim\bar{H}_0+\Ez+\Ez^3,\notag
\end{align}
  so that in particular $u$ is algebraically close to $\N_N(\frac{\ell_{0}}{2})$ at $s_1$ in the sense that
  \begin{align}
    \norm{u(s_1)-v(s_1)}_{H^1}^2+\E(s_1)\lesssim \Lambda^{-2}.\label{uve2}
  \end{align}
  \item[(ii)] On $[s_1,\delta^{-1}]$ there holds
 \begin{align}
  \norm{u(t)-v(t)}_{H^1}^2+ \E(t)&\lesssim \Lambda^{-2}\exp\left(\frac{-2t}{C_{\mathrm{ed}}\Lambda^2}\right)+\delta,\label{t.2bb}\\
   \abs{c(t)-c(s_1)}&\lesssim \Lambda^{-\frac{1}{2}},\label{t.4b}\\
   \norm{u(t)-u(s_1)}_{\dot{H}^{-1}}&\lesssim 1\label{t.5b},
 \end{align}
 and on $[s_1+1,\delta^{-1}]$ the solution has exactly $N$ simple zeros and the dissipation is of order $\Lambda^{-2}$.
  Above the constants in $\lesssim$ may depend on $C_H,C_E$, and $C_1$. Below all constants are universal.

  \item[(iii)] There exists $s_2\lesssim C(C_{H},C_{E},C_{1})\ell(0)\Lambda^2$ such that on the exponentially long time interval $[s_2,\delta^{-1}]$
  the solution is exponentially close to $\N_N(\frac{\ell_{0}}{2})$ and slowly evolving in the sense that
  \begin{align}
     \norm{u(t)-v(t)}_{H^1}^2+\E(t)&\lesssim \Lambda^2\delta^2,\label{uve1}\\
    \abs{c(t)-c(s)}&\lesssim \ell(0)^{-\frac{1}{2}}\delta\left(\abs{t-s}+\Lambda^{2}\right)\label{cest},\\
    \norm{u(t)-u(s)}_{\dot{H}^{-1}}&\lesssim \delta(\abs{t-s}+\Lambda^2),\label{uest}
  \end{align}
  and the dissipation is exponentially small of order $\Lambda^{2}\delta^{2}$.
\end{enumerate}
\end{theorem}
\begin{remark}
	In the proof we choose the lengthscale $\ell_0$ between zeros of $\bar{v}$ sufficiently large with respect to $C_H,\,C_E$ so that the lengthscale $\ell(0)$ between zeros of the initial data is of the same order, i.e., $\ell(0)\sim \ell_{0}$; see \eqref{scaleell}.
\end{remark}
\begin{remark}
	Clearly the results of the theorem \ref{t:main1} also hold up to any time of order $\delta^{-1}$ but for simplicity, we have stated the result for $t\leq\delta^{-1}$.
\end{remark}

\begin{remark}\label{rem:shortcoming}
  There are at least two (related) ways in which one could hope to improve the result.
    One can improve the assumption \eqref{ellbig}, via an iterative argument, to $\ell_0\sim \Lambda^{\frac{1}{k}}$ for $k\in \mathbb{N}$, at the cost of a $k$-dependence in the scaling bounds. However in principle one expects the result to hold as long as $\ell$ is sufficiently large, independent of the system size. Unfortunately such an estimate is out of reach via our current method. The second limitation is our assumption that $\bar{H}_0$ is order one. Generically, before a collision, the $\dot{H}^{-1}$ distance to the slow manifold with $N-2$ layers scales with $\ell$. These issues motivate work in progress \cite{OSW}, which offers new insight into the relaxation problem on the line.
\end{remark}

We are particularly interested in placing only weak conditions on the initial data. However  in the case of  initial data that are already close to the slow manifold, a direct application of the metastability framework yields a lower bound on the exponential timescale.
\begin{corollary}\label{cor:wid}
For initial data that are well prepared in the sense that
$
\Ez\lesssim\Lambda^{-2},
$
not much happens for a long time in the sense that
\begin{align*}
  \norm{u(t)-u(0)}_{\dot{H}^{-1}}^{2}\lesssim \Lambda\Ez^{\frac{1}{2}}+\delta \left(t+\Lambda^{2}\right).
\end{align*}
In particular, there exists $\eps>0$ such that $\Ez\leq\eps \Lambda^{-2}$ implies
\begin{align*}
  \inf\left\{t>0\colon \norm{u(t)-u(0)}_{\dot{H}^{-1}}=1 \right\}\gtrsim \delta^{-1}\sim \Lambda^\frac{1}{2}\exp\bigl(C_G\ell(0)\bigr).
\end{align*}
\end{corollary}
The theorem and corollary are proved in subsection \ref{ss:main}.
\begin{notation}\label{notation:sim}
  Throughout the paper we use the notation
  \begin{align*}
    A\lesssim B
  \end{align*}
  if there exists a universal constant $C\in(0,\infty)$ depending at most on the potential $G$, such that
  $
    A\leq C\,B
  $
  for $\ell$ and/or $\Lambda$ large (the relevant case being clear from the context). In the case of nonuniversal constants, such as dependence on $C_E$ in lemma \ref{L1b}, we will indicate this by, for example, notation such as $\lesssim_{C_E}$ or an explicit remark, as in theorem \ref{t:main1}.
  Moreover, we occasionally write
  \begin{align*}
  	A\ll B
  \end{align*}
  if for every $c>0$ there is an $L>0$ such that for $\ell\geq L$ or $\Lambda\geq L$ we have that
  $
  	A\leq c B.
  $

  We occasionally use $C$ to represent a constant whose definition may change from line to line and for exponentially small terms when we are not interested in a sharp constant.
\end{notation}
 \subsection{Previous results in the literature}\label{ss:lit}
Because some of the same methods were used for the Cahn--Hilliard equation, we begin with a brief summary of the analysis of the Allen--Cahn equation. Metastability of the Allen--Cahn equation was explored in the seminal works of Carr and Pego \cite{CP} and Fusco and Hale \cite{FH}. In \cite{CP}, a careful spectral analysis is used to establish, moreover, that initial data starting exponentially close to an appropriately defined slow manifold stays exponentially close for a time that is of exponential order in $\ell$ (the distance between the two closest zeros)---with the optimal constant $C_G$.

Subsequently, Bronsard and Kohn introduced a natural and elementary energy method that reduced the restriction on the initial data---it was only required to be algebraically close to the slow manifold---albeit at the expense of weakening the result---the solution is only shown to be trapped close to the slow manifold for an algebraically long time. The later analysis by Chen \cite{C} was exhaustive: A complete characterization of the evolution was established, including the sharp exponential constant for the exponentially slow phase; in this work, classical PDE-techniques (including maximum principles) are used. In \cite{OR}, an abstract metastability framework is introduced and used, together with a buckling argument, to give a different proof of the fact that, for initial data that is order one away from the slow manifold, exponential closeness to the slow manifold is \emph{generated} and subsequently \emph{propagated} for an exponentially long time, again with the sharp constant $C_G$ in the exponential.

Turning to the Cahn--Hilliard equation, Bronsard and Hilhorst \cite{BH} apply the method of \cite{BK} to the one-dimensional Cahn--Hilliard equation to show that initial data algebraically close to the slow manifold remains trapped nearby for an algebraically long time. In another application of the method of \cite{BK}, Grant \cite{G} constructs solutions of Cahn-Morral systems that remain close to the slow manifold for an exponential period of time.
Using the method of spectral estimates from \cite{ABF}, Bates and Xun \cite{BX1,BX2} were able to show that for initial data
algebraically close to the slow manifold (roughly speaking, that the $H^{1}$ norm of $u-v$ is sufficiently small with respect to $\Lambda^{-3}$, in our notation), the solution remains trapped for an
exponentially long time, with almost the sharp constant in the exponential.

\section{Method and proof of main result}
As explained in the introduction, we will use the relaxation framework from \cite{OW} to control the initial phase of energy relaxation and the metastability framework of \cite{OR} to control the second phase of energy relaxation and establish slow motion. We begin by explaining these tools, deferring the proofs. Then in subsection \ref{ss:main}, we show how to combine these tools to prove the main theorem.

In both phases, an important role will be played by the following nonlinear energy and dissipation estimates.
\begin{lemma}\label{l:eedraw}
	For any $C_H,\,C_E\in (0,\infty)$, there exists  $\ell_1\in (0,\infty)$ such that, for any $\ell\geq \ell_1$, $N\in\mathbb{N}$, and $u\in \M_N(\ell,C_H,C_E)$, there holds:
	\begin{align}
		\norm{u-v}_{H^1}^2&\lesssim_{\depEH}\E\lesssim_{\depE} \norm{u-v}_{H^1}^2,\label{L1b.1}\\
		\norm{u-v}_{\dot{H}^1}^2 &\lesssim_{\depEH}   D.\label{L1b.2}
\end{align}
	Moreover, there is a $\gamma>0$ such that for $\E(u)\leq \gamma$  \eqref{L1b.1} and \eqref{L1b.2} hold with  constants that are independent of $C_{H}$ and $C_{E}$.
\end{lemma}

\subsection{Relaxation framework}\label{ss:relax}
In the first part of the proof, we will use the relaxation framework from \cite{OW}, which requires appropriate  algebraic and differential relationships among the relevant quantities.
We will use lemma \ref{l:eedraw} together with the following algebraic relationships.
\begin{lemma}\label{L1b}
	There is a constant $C\in (0,1)$ such that
	for any $C_H,\,C_E\in (0,\infty)$, there exists  $\ell_1\in (0,\infty)$ such that, for any $\ell\geq \ell_1$, $N\in\mathbb{N}$, and $u\in \M_N(\ell,C_H,C_E)$, there holds:
	\begin{align}
		\abs*{c-\bar{c}}^2&\lesssim_{\depEH} \left(\bar{H}\E\right)^\frac{1}{2}+(|c-\bar{c}|+1)\E+\frac{\bar{H}}{\ell}+\exp(-C\ell),\label{L1b.3}\\
		\E&\lesssim_{\depEH}  \left(\bar{H} D\right)^\frac{1}{2}+(|c-\bar{c}|+1)^2 D+\exp(-C\ell),\label{L1b.4}\\
		\big| \E-\eb\big|&\lesssim \, \exp(-C\ell).\label{L1b.5}
	\end{align}
\end{lemma}
The proof of lemma \ref{L1b} is given in subsections \ref{ss:mainproofs}-\ref{ss:app_diss}.
We will use in addition the following differential relationships.
\begin{lemma}\label{L2b}
	There is a constant $C\in (0,\infty)$ such that
	for any $C_H,\,C_E\in (0,\infty)$, there exists  $\ell_1 \in (0,\infty)$ such that, for any $\ell\geq \ell_1$, $N\in\mathbb{N}$, and $u\in \M_N(\ell,C_H,C_E)$ that is
	a solution of the Cahn--Hilliard equation \eqref{ch}, there holds
	\begin{align}
		\frac{\dd\eb}{\dd t}&=-D,\label{L2b.1}\\
		\frac{\dd\bar{H}}{\dd t}&\lesssim_{\depEH}\left((\abs{c-\bar{c}}+1)\left(c-\bar{c}\right)^2 D\right)^\frac{1}{2}+\E^\frac{3}{4}D^\frac{1}{4}+\exp(-C\ell).\label{L2b.2}
	\end{align}
\end{lemma}
The proof of lemma \ref{L2b} is given in subsection \ref{ss:mainproofs}.

As in \cite{OW}, the algebraic and differential relationships among the central quantities are linked via an ODE lemma. We formulate the lemma in terms of the notation with which it will be applied in our work.
\begin{lemma}[Lemma 1.5, \cite{OW}]\label{l:diffeq}
Suppose that $\eb(t)\geq 0$, $D(t)\geq 0$, $\bar{H}(t)\geq 0$ and $(c(t)-\bar{c})^2\geq 0$ for $t\in[0,t_*]$ are related by the differential inequalities
\begin{equation}\label{d1}
\frac{\dd\eb}{\dd t}=-D,\quad\frac{\dd\bar{H}}{\dd t}\lesssim
c_*^\frac{1}{2}\Big[\big((c-\bar{c})^2D\big)^\frac{1}{2}+\eb^\frac{3}{4}D^\frac{1}{4}\Big],
\end{equation}
and by the algebraic inequalities
\begin{equation}\label{a1}
\eb\lesssim(\bar{H} D)^\frac{1}{2}+c_*^2D\quad\mbox{and}\quad (c-\bar{c})^2\lesssim(\bar{H}
\eb)^\frac{1}{2}+c_*\,\eb,
\end{equation}
where $c_*$ is a fixed number.
Then on $[0,t_*]$ there holds
\begin{eqnarray}
\eb(t)&\le&\eb_0,\label{n.3}\\
\eb(t)&\lesssim&\left(\bar{H}_0+c_*^2\eb_0\right)t^{-1},\label{L3.3a}\\
(c(t)-\bar{c})^2&\lesssim&\left(\bar{H}_0+c_*^2\eb_0\right)^\frac{1}{2}\eb(t)^\frac{1}{2},\label{L3.3b}\\
\bar{H}(t)&\lesssim& \bar{H}_0+c_*^2\eb_0,\label{L3.3c}
\end{eqnarray}
where $\eb_0=\eb(0)$ and $\bar{H}_0=\bar{H}(0)$.
\end{lemma}
\begin{remark}
For the proof we refer to \cite{OW}.  The only difference is that \eqref{L3.3b} takes a slightly different form, but this can be read off directly from the proof of \cite[lemma 1.5]{OW}.
\end{remark}

\subsection{Metastability framework}\label{ss:meta}

After the initial energy relaxation, our proof is based on the metastability framework developed in \cite{OR}.
It is convenient to use the weak norm defined via
\begin{align}\label{weaknorm}
  \norm{v}_{\N}^2:=\int v \left(\left(\frac{1}{\ell^2}-\partial_{xx}\right)^{-1} v\right)\,\dd x
\end{align}
so that we do not have to restrict our slow manifold to functions with mean zero. Here $\ell$ is a lengthscale that is fixed in the proof of theorem \ref{t:main1}.
The following lemma establishes that this norm is indeed weaker than the $\dot{H}^{-1}$ norm and relates shifts of zeros of energy optimal profiles to their distance in the weak norm.
\begin{lemma}\label{l:vdist}
For all $w\in L^2\left(\torus_{\Lambda}\right)$, there holds
\begin{align}
 \norm{w}_{\N}\leq\min\left\{\norm{w}_{\dot{H}^{-1}},\, \ell \,\norm{w}_{L^2} \right\}. \label{v.2}
\end{align}
For any $\NN\in\mathbb{N}$ there exist $\ell_1,C_2\in(1,\infty)$ with the following property. For any $\ell\geq\ell_1$, $N\leq \NN$, and $v,\,\tilde{v}\in\N_N(\ell)$ with zeros $c,\,\tilde{c}$ such that $|c-\tilde{c}|\leq \frac{\ell}{C_2}$, there holds
  \begin{align}
|c-\tilde{c}|^2\lesssim \frac{1}{\ell}\norm{v-\tilde{v}}_{\N}^2.\label{v.1}
  \end{align}
\end{lemma}
The proof of lemma \ref{l:vdist} is given in subsection \ref{ss:weak}. Our two conditions for metastability now follow.
The first condition is an energy--energy--dissipation (EED) relationship.
\begin{lemma}\label{l:eed}
For any $C_H,\,C_E\in (0,\infty)$, there exist  $\ell_1,\, C_{\mathrm{ed}}\in (0,\infty)$ such that, for any $\ell\geq \ell_1$, $N\in\mathbb{N}$, and $u\in \M_N(\ell,C_H,C_E)$ there holds
\begin{align}
  \frac{1}{2C_{\mathrm{ed}}\Lambda^2}\norm{u-v}_{\N}^2\leq \E\leq \frac{C_{\mathrm{ed}} \Lambda^2}{2} D.\label{eed}
\end{align}
Moreover, there is a $\gamma>0$ such that for $\E(u)\leq \gamma$  \eqref{eed} holds with a constant $C_{\mathrm{ed}}$ that is independent of $C_{H}$ and $C_{E}$.
\end{lemma}
\begin{proof}
  The proof follows from lemma \ref{l:eedraw} together with \eqref{v.2} and the Poincar\'e inequality.
\end{proof}
Our second condition is the Lipschitz condition on $\N_N(\ell)$.
\begin{lemma}\label{l:lip}
	There is $\ell_1\in(0,\infty)$ with the following property. For any $\ell\geq \ell_{1}$, $N\in\mathbb{N}$, and $v,\tilde v\in\N_N(\ell)$, there holds
	\begin{align}\label{lipschitz2}
		\abs*{E(v)-E(\tilde v)}\lesssim \delta \norm{v-\tilde v}_{\N}\quad\text{with}\quad\delta:= N\ell^{-\frac{1}{2}}\exp(-C_G\ell) .
	\end{align}
\end{lemma}
\begin{proof}
The proof follows by applying
\begin{align}\label{lipschitz1}
		\abs*{E(v)-E(\tilde v)}\lesssim N\exp\left(-C_G\ell\right) \abs{c-\tilde c}.
	\end{align}
(cf. \cite[Lemma 3.1]{OR}) and  \eqref{v.1}.	
\end{proof}

In order to deduce metastability via the weak norm $\norm{\cdot}_{\N}$, we need to slightly modify theorem 1.1 from \cite{OR}; proposition \ref{prop:orweak} below carries this out. An additional, more significant issue is that \cite[theorem 1.1]{OR} implicitly assumes integrability of $E(v(t))$. While this was elementary for the application to the Allen-Cahn equation, since we assumed simple zeros of the initial condition and this property is preserved by the flow (for an exponentially long time), this is not the case for the Cahn-Hilliard equation. Indeed, even an initial condition with simple zeros may develop spurious zeros during the initial energy relaxation.
In one spatial dimension this difficulty can be readily circumvented by making a measurable selection of the zeros.
We will instead introduce a metastability result that makes no assumption of integrability, which may be of interest for more complicated applications.

Proposition \ref{prop:orweak2} captures metastability by establishing exponential in time convergence of $\eb$ (instead of $\E$). In so doing, one obtains the sharp factor of $2$ in \eqref{u.first}---at the expense of an error term that goes like $\delta$ (instead of $\delta^2$); compare to \eqref{u.first.2}. In a postprocessing step after applying proposition \ref{prop:orweak2}, we deduce that the zeros are in fact simple on $[s_1+1,T]$ (cf. lemma \ref{l:diss2}) and hence obtain via proposition \ref{prop:orweak} that the solution indeed enters a $\delta^2$ neighborhood of the slow manifold.

We now state our two metastability propositions. We will apply them to the Cahn--Hilliard evolution with $\N=\N_N(\frac{\ell_{0}}{2})$, $\norm{\cdot}_0=\norm{\cdot}_\N$, and $\norm{\cdot}_1=\norm{\cdot}_{\dot{H}^{-1}}$.
\begin{proposition}[Metastability without assuming integrability of $E(v(t))$]\label{prop:orweak2}
	Let $(X_{0},\norm{\cdot}_{0})$ be a Banach space and $(X_{1},\norm{\cdot}_{1})$ be a Hilbert space with
	\begin{align}
		X_{1}\subset X_{0}\qquad\text{and}\qquad
		\norm{u}_{0}\leq \norm{u}_{1}.\label{u3.4}
	\end{align}
	For a differentiable functional $E:X_{1}\to\R$, consider the gradient flow
	\begin{align}\label{gfeq2}
		\dot u=-\nabla E(u)\qquad u(0)=u_0.
	\end{align}
	Let $\mathcal{M}\subset X_{1}$ and suppose that there is a set $\mathcal{N}\subset X_{0}$ with the following properties:
	\begin{enumerate}
		\item[(i)]
			For every $u\in\mathcal{M}$ there is a $v\in\mathcal{N}$ such that
			\begin{align}\label{eed2.2}
				\frac{1}{2}\norm{u-v}_{0}^{2}\leq E(u)-E(v)
				\leq \frac{1}{2}\norm{\nabla E(u)}_{1}^{2}.
			\end{align}
		\item[(ii)]
			There is a constant $\delta\in (0,1)$ such that for all $v_{1},v_{2}\in\mathcal{N}$ we have
			\begin{align}\label{lipschitz.2}
				\abs{E(v_{1})-E(v_{2})}\leq \delta \norm{v_{1}-v_{2}}_{0}.
			\end{align}
	\end{enumerate}
	Suppose that for $t\in[0,T]$, the solution of \eqref{gfeq2} satisfies $u(t)\in \mathcal{M}$ and let $v(t)\in\mathcal{N}$ denote a function associated to the solution at time $t$ as in (i). Define the initial quantity
$
  e_0:=E\bigl(u(0)\bigr)-E\bigl(v(0)\bigr)
$
	and suppose that $\delta$ is small enough so that
$\ln(e_0^{\frac{1}{2}}/\delta)\leq \delta^{-1}.$	
Then for all $t\leq \delta^{-1}$ there holds
	\begin{align}
			\abs*{E\bigl(u(t)\bigr)-E\bigl(v(0)\bigr)}\lesssim \exp\bigl(-2t\bigr)e_0+\delta\left(e_0^{\frac{1}{2}}+e_0^{\frac{1}{4}}\right).\label{u.first}
	\end{align}
	Furthermore, changes up to time $T_1:=\max\{0,2\ln\left( \frac{e_0}{\delta}\right)\}$  are controlled by
  \begin{align}
  \sup_{t\leq T_1}\Big(\norm{u(t)-u(0)}_1+ \norm{v(t)-v(0)}_0\Big) \, \lesssim e_0^{\frac{1}{2}}+e_0^{\frac{1}{4}},\label{13.s2}
\end{align}
and the energy gap at $t=T_1$ is small:
\begin{align}
  \abs*{E\bigl(u(T_{1})\bigr)-E\bigl(v(0)\bigr)}+\abs{E\bigl(u(T_{1})\bigr)-E\bigl(v(T_{1})\bigr)}\lesssim \delta\left(e_0^{\frac{1}{2}}+e_0^{\frac{1}{4}}\right).\label{13.s1}
\end{align}
For times $T_1\leq t\leq \delta^{-1}$ the motion is slow in the sense that
		\begin{align}
		\norm{u(t)-u(T_1)}_{1}^2&\lesssim \delta(t-T_1)\left(e_0^{\frac{1}{2}}+e_0^{\frac{1}{4}}\right),\label{diffu}\\
\norm{v(t)-v(T_1)}_0^2&\lesssim  \delta(1+t-T_1)\left({e_0^{\frac{1}{2}}}+e_0^{\frac{1}{4}}\right)+\delta^2,\label{diffv}
	\end{align}
and the solution is trapped near the slow manifold via
\begin{align}
\norm{u(t)-v(t)}_0^2+\abs*{E\bigl(u(t)\bigr)-E\bigl(v(t)\bigr)}\lesssim \delta\left({e_0^{\frac{1}{2}}}+e_0^{\frac{1}{4}} +\norm{v(t)-v(T_1)}_0\right).\notag
\end{align}
\end{proposition}

\begin{proposition}[Metastability under weak norm condition]\label{prop:orweak}
	Let $X_{0}$, $X_{1}$, $E$, $\mathcal{M}$, $\mathcal{N}$, $u$, $v$, $T$, and $\delta$ be as in proposition \ref{prop:orweak2}.
Suppose moreover that $t\mapsto E(v(t))$ is integrable on $[0,T]$.
	Then for every $\epsilon\in (0,1)$ there is a constant $C_{\epsilon}$ such that
	\begin{align}\label{u.first.2}
		\begin{split}
			\MoveEqLeft
			\norm{u(t)-v(t)}_{0}^{2}+E\bigl(u(t)\bigr)-E\bigl(v(t)\bigr)\\
			&\lesssim \exp\bigl(-(2-\epsilon)t\bigr)\bigl(E(u_{0})-E(v_{0})\bigr)+C_{\epsilon}\delta^{2}.
		\end{split}
	\end{align}
	Furthermore, for $0<s<t\leq T$ there holds
	\begin{align}\label{slowmotion}
		\norm{u(t)-u(s)}_{1}\lesssim \Big(E\bigl(u(s)\bigr)-E\bigl(v(s)\bigr)\Big)^{\frac{1}{2}}+\delta(t-s+1).
	\end{align}
\end{proposition}
The proofs of both propositions are given in section \ref{S:meta}.
\subsection{Postprocessing: small dissipation and simple zeros}\label{ss:post}
In order to deduce that the solution has simple zeros after the energy has relaxed, we derive the following differential inequality for the dissipation.
\begin{lemma}\label{l:diss1}
  There exist $\ell_1,\eps\in\R^+$, such that if $v\in \N_N(\ell_1)$ for some $N\in\mathbb{N}$ and the solution $u$ of the Cahn-Hilliard equation \eqref{ch} satisfies $\norm{u-v}_\infty\leq \eps$, then
  \begin{align}
    \frac{\dd}{\dd t}D\lesssim D.\label{d1.down}
  \end{align}
\end{lemma}
We use this inequality to deduce from a small energy gap a bound on the dissipation.
\begin{lemma}\label{l:diss3}
There exist $\ell_1,\,\gamma\in\R^+$ such that if $\ell\geq \ell_1$ and $\eb\leq \gamma$ on $[s,t]$ for some $t\geq s+1$, then
\begin{align*}
  \max_{[s+1,t]}D\lesssim \gamma.
\end{align*}
\end{lemma}
Smallness of the the dissipation then implies simple zeros.
\begin{lemma}\label{l:diss2}
  There  exist $\ell_1,\eps\in\R^+$ such that $v\in \N_N(\ell_1)$ and $\E+D\leq \eps$ implies that $u$ has exactly $N$ simple zeros.
\end{lemma}
We prove lemmas \ref{l:diss1}, \ref{l:diss3}, and \ref{l:diss2} in subsection \ref{ss:disspf}, below.
\subsection{Proof of  theorem \ref{t:main1} and corollary \ref{cor:wid}}\label{ss:main}
We are now ready to prove our main theorem.

\begin{proof}[Proof of theorem \ref{t:main1}]
We remark that according to \eqref{L1b.1}, it suffices for \eqref{egap1}, \eqref{t.2bb} and \eqref{uve1} to establish the upper bound for the energy gap.

\step{Step 0.} In this preliminary step we collect the facts that we will need about the motion of and distance between zeros. First we remark that, according to \eqref{ellbig}, we may choose $\Lambda_1$ big enough so that
\begin{align}
\ell_0\geq \frac{\Lambda_1}{C_1}\quad\text{is as large as necessary}
\end{align}
and in particular (a) so that $2\ell_0\geq \ell_1$ for the constant $\ell_1$ from the above auxiliary lemmas, and (b) so that lemma \ref{l:uv} gives control on the initial zeros of $u$ in the form
\begin{align}
  \abs{c(0)-\bar{c}}\leq\frac{\ell_0}{8},\quad \ell(0)\geq\frac{3\ell_0}{4},\label{cell}
\end{align}
where $\ell(t)$ denotes the distance between the closest consecutive zeros in $c(t)$.

The proof will rely on a buckling argument. A priori, we will need to control changes in $\bar{H}(t)$ and $\ell(t)$. To this end, we
define the maximal displacement of the zeros as
\begin{align*}
  \Delta c(t):=\sup_{0\leq t'\leq t}|c(t')-c(0)|,
\end{align*}
introduce constants $C_3, C_4\in (1,\infty)$ (to be specified later) depending only on universal constants and $\bar{H}_0$, $E_0$, and define the times:
\begin{align}
  T_1&:=\inf\left\{t\geq 0\colon \bar{H}(t)\geq C_3\right\},\label{t1}\\
  T_2&:=\inf\left\{t\geq 0 \colon \Delta c(t)\geq C_4  \right\}.\label{t2}
\end{align}
Note that according to \eqref{t1} and $E(u(t))\leq E_{0}$, the function $u(t)\in \M_N(\ell_0,C_3,E_0)$ and hence,
by lemma \ref{l:uv},  $v(t)$ is well-defined and lemma \ref{L1b} applies for all $t\in[0,T_1]$.
We also choose $\ell_0$ so that
\begin{align}
\ell_{0}\geq\max\left\{8C_4,2C_{2}C_{4}\right\},\label{maychoose}
\end{align}
where $C_2$ is the maximum of the corresponding constants from lemmas \ref{l:vdist} and \ref{l:lip}.
On the one hand, we have for $t\leq T_2$ the estimate
\begin{align}
	\ell(t)\geq \ell(0)-|\Delta c|\geq \ell(0)-C_4\geq \frac{3\ell(0)}{4}\overset{\eqref{cell}}\geq \frac{\ell_{0}}{2},\label{ell34}
\end{align}
which according to \eqref{ellbig}, the second item in \eqref{cell}, and \eqref{ell34}, implies
\begin{align}\label{scaleell}
	\ell_0\sim \ell(0)\sim \ell(t)\sim \Lambda.
\end{align}
Note that by \eqref{ell34} we have that $v(t)\in \N_{N}(\frac{\ell_{0}}{2})$.
On the other hand, we control the motion of the zeros  for $t\leq T_2$ in terms of the total interval length via
\begin{align}
  \abs{c(t)-c(0)}\leq C_4\overset{\eqref{maychoose}}\leq \frac{\ell_{0}}{2C_2},\label{clt}
\end{align}
so that lemmas \ref{l:vdist} and \ref{l:lip} for $\ell=\frac{\ell_{0}}{2}$ in the definition of $\norm{\cdot}_{\N}$ apply on $t\leq T_2$.

We now control two phases of the evolution. In the initial phase, we will use a lower bound on the energy gap and the relaxation framework (in particular lemma \ref{l:diffeq}, above) to establish algebraic decay of $\E$ and a bound on $\bar{H}$.
To this end, we define
\begin{align}
  T_3:=\inf\left\{t\geq 0\colon \frac{\bar{H}(t)+1}{\Lambda^2} \geq  \E(t) \right\}\label{t3}
\end{align}
and
$
  T_*:=\min\{T_1, T_2, T_3\}.
$
We will show in step 1 that
\begin{align}
  T_*<\min\{T_1, T_2\},\quad T_*\lesssim \left(\bar{H}_0+\Ez+\Ez^{3}\right)\Lambda^2,\quad\text{and}\quad \E(T_*)\lesssim \Lambda^{-2}.\label{boots}
\end{align}
The timescale $s_1$ in the statement of the theorem is then defined as
\begin{align*}
  s_1:=T_*.
\end{align*}
Once the energy gap has become small, we will use the solution at $T_{*}$ as \emph{initial data} for the EED framework of \cite{OR} (in particular proposition \ref{prop:orweak2}) and establish and make use of the exponential decay in time of the energy gap
on the interval $[T_*,T_{**}]$ for
\begin{align}
  T_{**}:=\min\left\{T_1,T_2,  \Lambda^{\frac{1}{2}}\exp(C_G\ell_0)\right\}.\label{t4}
\end{align}
As a consequence we will obtain in step 2 that
\begin{align}
  T_{**}<\min\{T_1,\,T_2\}\quad\text{and hence}\quad T_{**}= \Lambda^{\frac{1}{2}}\exp(C_G\ell_0).\label{boots2}
\end{align}

\step{Step 1.} Here we consider the evolution on $[0,T_*]$. We set for convenience
\begin{align}
  c_*:=\sup_{t\in[0,T_*]}\abs*{c(t)-\bar{c}}+1.\label{cstar}
\end{align}
Notice that
\begin{align}\label{cstar2}
	1\leq c_{*}\leq \Lambda+1
\end{align}
and that, by definition of $T_*$ and lemma \ref{l:uv}, $c_*$ is bounded in terms of $C_3$, $C_4$.
We will show that, in fact, $c_*$ is bounded independently of $C_3$, $C_4$, and $\Lambda$, and that there holds
\begin{eqnarray}
\E(t)&\le&2\Ez,\label{t.1}\\
\E(t)&\lesssim&\left(\bar{H}_0+\Ez+\Ez^3\right)t^{-1},\label{t.2}\\
(c(t)-\bar{c})^2&\lesssim&\left(\bar{H}_0+\Ez+\Ez^3\right)^\frac{1}{2}\E(t)^{\frac{1}{2}},\label{t.4}\\
\bar{H}(t)&\lesssim& \bar{H}_0+\Ez+\Ez^3.\label{t.5}
\end{eqnarray}
Notice that we may assume without loss of generality that
\begin{align}
  \Ez\geq \frac{\bar{H}_0+1}{\Lambda^2},\label{e0a}
\end{align}
since otherwise we may set $T_*=0$.
As a consequence of \eqref{L1b.5} and \eqref{scaleell}, we have
\begin{align}\label{deltagap}
	|\E-\eb|\lesssim\exp(-C\Lambda) \qquad\text{on $[0,T_*]$}.
\end{align}
Using \eqref{t3} and \eqref{L2b.1} leads to
\begin{align}\label{scalegap}
	\E\sim\eb, \qquad
	\eb\gtrsim \Lambda^{-2},\qquad\text{and}\qquad
	\eb\leq\eb(0)
	\qquad\text{on $[0,T_*]$},
\end{align}
which for $\Lambda$ large, implies \eqref{t.1}.

We turn now to verifying the conditions of lemma \ref{l:diffeq} for the Cahn--Hilliard evolution on $[0,T_*]$. We allow a dependence on $C_3$ and $C_E$ and at the end of the step verify that $C_3$ can be chosen to depend only on $C_H$ and $C_E$.
The first item in \eqref{d1} is \eqref{L2b.1}.
The first item in \eqref{a1} hence follows from  \eqref{L1b.4} together with \eqref{scaleell} and
\begin{align}
	\exp(-C\Lambda)\overset{\eqref{scalegap}}\ll \E.\label{stern}
\end{align}
The second item in \eqref{a1} follows from \eqref{L1b.3}, \eqref{scaleell}, \eqref{scalegap}, \eqref{stern}, and
\begin{align}
  \frac{\bar{H}}{\Lambda}\overset{\eqref{t3}}\leq  (\bar{H}\E)^{\frac{1}{2}}.\label{hle}
\end{align}
It remains to establish the second item in \eqref{d1}, which will follow from \eqref{L2b.2}, \eqref{scaleell}, and \eqref{scalegap}. To this end, it suffices to observe that
\begin{align}
  D^{-1}\overset{\eqref{L1b.4}}\lesssim \bar{H}\E^{-2}+c_*^2\E^{-1}\overset{\eqref{t3}}\lesssim \frac{\Lambda^4}{\bar{H}+1}\lesssim \Lambda^4 \qquad\text{on $[0,T_*]$}.\label{diss1}
\end{align}
Combining this estimate with \eqref{scalegap} implies
\begin{align*}
  \exp(-C\Lambda)\lesssim \eb^\frac{3}{4}D^\frac{1}{4}\Lambda^\frac{5}{2}\exp(-C\Lambda)\leq \eb^\frac{3}{4}D^\frac{1}{4}
\end{align*}
for $\Lambda$ sufficiently large.

Hence we may apply lemma \ref{l:diffeq}. Note that according to \eqref{L3.3b}, \eqref{n.3}, \eqref{scalegap}, and Young's inequality, there hold
\begin{align}
  c_*^2\lesssim  \left(\bar{H}_0\Ez\right)^\frac{1}{2}+\Ez^2+1,\qquad \bar{H}_0+c_*^2\Ez\lesssim \bar{H}_0+\Ez+\Ez^3.\label{box1}
\end{align}
Consequently estimates \eqref{t.2}-\eqref{t.5} follow from
\eqref{L3.3a}-\eqref{L3.3c} and \eqref{scalegap}.

Finally, we use  estimates \eqref{t.4} to \eqref{t.5} (together with lemma \ref{l:uv}) to deduce that \eqref{boots} holds for $C_3$ and $C_4$ sufficiently large with respect to $C_H$ and $C_E$. From \eqref{boots} and \eqref{L1b.1} we deduce \eqref{uve2} for $s_1:=T_*$. The bound on $s_1$ follows from \eqref{t.2}.\\

\step{Step 2.} Here we consider the evolution on $[T_*, T_{**}]$ and show \eqref{t.2bb},
 \eqref{t.4b}, and \eqref{t.5b}.
Analogously to in step 1, we obtain \eqref{boots2} from \eqref{t.4b} and \eqref{t.5b} for $C_3$ and $C_4$ sufficiently large.
We note for reference below that, according to \eqref{eed} and \eqref{uve2}, there holds
\begin{align}
	\frac{1}{\Lambda^2}\norm{u(T_*)-v(T_*)}_{\N}^2+  \E(T_*)\lesssim \frac{1}{\Lambda^2}.\label{etstar}
\end{align}
We will carry out this step by applying the metastability framework of proposition \ref{prop:orweak2}.
To this end, we note that by \eqref{uve2}, \eqref{L2b.1}, the Lipschitz condition \eqref{lipschitz1}, and \eqref{scaleell} we may choose $\Lambda_{1}$ large enough
to ensure that the energy gap is at $T_*$ is smaller than the constant $\gamma$ from lemma \ref{l:eedraw}.
According to \eqref{ell34}, definition of $T_{**}$, and lemma \ref{l:eed}, one has \eqref{eed} on $[T_*,T_{**}]$, with a universal constant $C_{\mathrm{ed}}$.
Furthermore, by lemma \ref{l:lip} and \eqref{clt}, one also has a Lipschitz condition on $[T_*,T_{**}]$ of the form
\begin{align*}
  \abs*{E\bigl(v(t)\bigr)-E\bigl(v(s)\bigr)}\lesssim \hat{\delta} \norm{v(t)-v(s)}_{\N},
\end{align*}
with
\begin{align}
  \hat{\delta}\lesssim \min\{\ell(t),\ell(s)\}^{-\frac{1}{2}}\exp\bigl(-C_G\min\{\ell(t),\ell(s)\}\bigr).\label{delinit}
\end{align}
Estimating the distances between zeros via the additive estimate \eqref{ell34} and the multiplicative estimate \eqref{scaleell} leads to
\begin{align}\label{deodda}
\hat{\delta}\lesssim \Lambda^{-\frac{1}{2}}\exp\bigl(-C_G\ell(0)\bigr)\exp\bigl(C\Delta c(t)\bigr)
 \overset{\eqref{t2}} \lesssim \Lambda^{-\frac{1}{2}}\exp\bigl(-C_G\ell(0)\bigr)\exp(CC_4).
\end{align}

To apply proposition \ref{prop:orweak2}, we rescale time and energy via $\tilde{t}=t/(C_\mathrm{ed}\Lambda^2)$, $\tilde{E}=(C_{\mathrm{ed}}\Lambda^2) E$ and deduce for
$t\in [T_*, T_{**}]$ that
\begin{align}
	\abs*{E\bigl(u(t)\bigr)-E\bigl(v(T_{*})\bigr)}&\lesssim  \Lambda^{-2}\exp\left(-\frac{2(t-T_{*})}{C_{\mathrm{ed}}\Lambda^2}\right)
	+\hat{\delta}\label{ephase2.b}
\end{align}
and for all $T_*\leq s\leq t\leq T_{**}$ that
\begin{align}
\norm{v(t)-v(s)}_{\N}+\norm{u(t)-u(s)}_{\dot{H}^{-1}}\lesssim  \bigl(1+\hat{\delta}t\bigr)^{\frac{1}{2}}.
\label{slowub.b}
\end{align}
Our next step is to develop a $C_{3}$- and $C_4$-independent bound on $\Delta c$ and, consequently, on $\hat{\delta}$. By \eqref{t.5} and \eqref{slowub.b} we then also obtain
a $C_{3}$ and $C_{4}$ independent bound on $\bar{H}$.
To this end we will use continuity of solutions in $\dot{H}^{-1}$ and the estimate
\begin{align}
\norm{u(t)-u(T_*)}_{\dot{H}^{-1}}\overset{\eqref{slowub.b},\eqref{deodda}, \eqref{t4}}\lesssim \Bigl(1+\exp\bigl(C\abs{c(t)-c(T_*)}\bigr)\Bigr)^{\frac{1}{2}}.\label{h1.1.b}
\end{align}
Now we use lemma \ref{l:vdist}, \eqref{scaleell}, and \eqref{etstar} to bound
\begin{eqnarray}
\lefteqn{  \abs{c(t)-c(T_*)}}\notag\\
&\lesssim&\Lambda^{-\frac{1}{2}}\norm{v(t)-v(T_*)}_\N\notag\\
  &\leq& \Lambda^{-\frac{1}{2}}\left(\norm{v(t)-u(t)}_\N+\norm{u(t)-u(T_*)}_\N+\norm{u(T_*)-v(T_*)}_\N\right)\notag\\
  &\overset{\eqref{etstar},\eqref{v.2}}{\lesssim} &\Lambda^{-\frac{1}{2}}\left( \norm{v(t)-u(t)}_\N +\norm{u(t)-u(T_*)}_{\dot{H}^{-1}}+1\right).\label{cstep.1}
\end{eqnarray}
To estimate the first term on the right-hand side, we use the simplistic estimate
\begin{eqnarray*}
  \norm{v(t)-u(t)}_\N^2&\overset{\eqref{eed}}\lesssim &\Lambda^2\Bigl(E\bigl(u(t)\bigr)-E\bigl(v(t)\bigr)\Bigr)\\
  &\overset{\eqref{L2b.1}}\leq & \Lambda^2 \Big(E\bigl(u(T_*)\bigr)-E\bigl(v(T_*)\bigr)+E\bigl(v(T_*)\bigr)-E\bigl(v(t)\bigr)\Big)\\
  &\overset{\eqref{etstar}}{\lesssim}& 1 +\abs*{c(t)-c(T_*)}\exp(-C\Lambda),
\end{eqnarray*}
with a nonsharp constant in the exponential.
Substituting into \eqref{cstep.1} and applying Young's inequality is enough to conclude
\begin{align}
\abs{c(t)-c(T_*)}\lesssim \Lambda^{-\frac{1}{2}}\left(\norm{u(t)-u(T_*)}_{\dot{H}^{-1}}+1\right).\label{h1.2}
\end{align}
Substituting \eqref{h1.2} into \eqref{h1.1.b} gives
\begin{align*}
  \norm{u(t)-u(T_*)}_{\dot{H}^{-1}}\lesssim 1+\exp\left(\frac{C \norm{u(t)-u(T_*)}_{\dot{H}^{-1}}}{\Lambda^{\frac{1}{2}}}\right).
\end{align*}
Continuity in $\dot{H}^{-1}$ and $\Lambda\gg 1$ implies
 \eqref{t.5b} and substituting \eqref{t.5b} back into \eqref{h1.2} leads to \eqref{t.4b}.
Finally we use the bound \eqref{t.4b} in the first inequality in \eqref{deodda} to estimate $\hat \delta$ by $\delta$ and use this estimate in \eqref{ephase2.b} and \eqref{slowub.b} to obtain uniform estimates on the energy gap $\E$.\\

\step{Step 3.}
As explained in subsection \ref{ss:meta}, we improve from a $\delta$ neighborhood of the slow manifold to a $\delta^2$ neighborhood in \eqref{uve1}
by establishing uniqueness of the $N$ zeros of $u$ and applying proposition \ref{prop:orweak}. Indeed, we just showed that the energy gap is algebraically small
on $[T_*,\delta^{-1}]$.
According to lemmas \ref{l:diss3} and \ref{l:diss2}, it follows that $u$ has $N$ simple zeros on $[T_*+1,\delta^{-1}]$. Hence $E(v(t))$ is measurable and, in particular, integrable, and we may apply proposition \ref{prop:orweak}. We rescale as above and apply \eqref{u.first.2} for $\epsilon=1$.
Choosing
\begin{align*}
s_2=C(C_{H}, C_{E}, C_{1})\Lambda^2\ln\left((\Lambda\delta)^{-2}\right)
\end{align*}
gives \eqref{uve1}. The other estimates follow from \eqref{slowmotion} and lemma \ref{l:vdist}.
\end{proof}

\begin{proof}[Proof of corollary \ref{cor:wid}]
	The proof is a minor modification of step 2 above.
\end{proof}

\section{Auxiliary proofs}\label{S:appendix}
\subsection{Preliminary lemmas}\label{ss:prelim}
\begin{proof}[Proof of lemma \ref{l:uv}]
	For simplicity, let $\bar x_{i}=0$. We will show that $u$ has a zero within a neighborhood of $\bar x_{i}$ of size	$L=C \max\{\bar{H}^{\frac{1}{3}},\bar{H}^{\frac{1}{5}}\}$.
	Note that for any $C>0$ we can choose $\ell_{1}$ so large that
	\begin{align*}
		C\max\left\{\bar{H}^{\frac{1}{3}},\bar{H}^{\frac{1}{5}}\right\}\leq C\max\left\{C_{H}^{\frac{1}{3}},C_{H}^{\frac{1}{5}}\right\}
		\leq \ell_{1}.
	\end{align*}
	Assume that $u$ does not have a zero in $(-3L,3L)$. Then $u$ has fixed, say negative, sign on this interval.
	Let $\bar v$ be positive on $(0,\bar x_{i+1})$.
	Furthermore, let $\eta$ be a cutoff function with
	\begin{align*}
		\eta\equiv 1 \text{ on }(L,2L),\quad
		\abs{\eta}\leq 1,\quad
		\eta\equiv 0\text{ on }\R\setminus (0,3L),\quad\text{and}\quad
		\abs{\eta_{x}}\lesssim L^{-1}.
	\end{align*}
	We estimate
	\begin{align}
		\begin{split}\label{L_inequality}
			\MoveEqLeft
			\int_{L}^{2L}\bar v\eta\,\dd x
			\leq \int(\bar v-u)\eta\,\dd x
			\leq \norm{u-\bar v}_{\dot H^{-1}}\norm{\eta_{x}}_{\dot H^{1}}
			\lesssim \left(\bar{H} L^{-1}\right)^{\frac{1}{2}}
		\end{split}
	\end{align}
	Define $x_{0}\in (0,\frac{\bar x_{i+1}}{2})$ such that $\bar v(x_{0})=\frac{1}{2}$. In case $L$ is so small that $x_{0}\in [\frac{3}{2}L,\frac{\bar x_{i+1}}{2})$, we obtain
	\begin{align*}
		\int_{L}^{2L}\bar v\eta\,\dd x\geq\int_{L}^{\min\{2L,x_{0}\}}\bar v\,\dd x
		\geqsim L^{2},
	\end{align*}
	since $\bar v\geqsim x$ on $(0,x_0)$.
	If $L$ is large enough so that $x_{0}\in (0,\frac{3}{2}L)$, on the other hand, we use monotonicity of $\bar v$ to estimate
	\begin{align*}
		\int_{L}^{2L}\bar v\eta\,\dd x\geqsim L.
	\end{align*}
	Combining the previous two estimates with \eqref{L_inequality} yields
	$
		L\lesssim \max\left\{\bar{H}^{\frac{1}{3}},\bar{H}^{\frac{1}{5}}\right\},
	$
	which contradicts the definition of $L$ for $C$ large enough.

	Because the energy of $v$ (and hence of $u$) on $(0,\ell)$ for $\ell$ large is bounded below by a positive constant, we obtain the energetic lower bound \eqref{ddistancezeroes}.
	Estimate \eqref{distance_uvc} then follows from
	\begin{eqnarray*}
		\norm{u-v}_{L^2}^2&\lesssim& \norm{u-\bar{v}}_{L^2}^2+\norm{\bar{v}-v}_{L^2}^2\\
		&\lesssim&\norm{u-\bar{v}}_{\dot H^{-1}}\left(\norm{u}_{\dot H^{1}}+\norm{\bar v}_{\dot H^{1}}\right)+\abs{\bar{c}-c}\norm{\bar{v}-v}_{L^{\infty}}^{2}\\
		&\overset{\eqref{ddistancezeroes}}\lesssim& \bar{H}^{\frac{1}{2}}\left(E(u)^{\frac{1}{2}}+1\right)+\bar{H}^{\frac{1}{3}}+\bar{H}^{\frac{1}{5}}.
	\end{eqnarray*}
\end{proof}
\subsection{Scaling of the weak norm}\label{ss:weak}
In this subsection we prove lemma \ref{l:vdist}. The main ingredient is the following approximation lemma, which says that functions of interest (which the reader can think of as differences of energy optimal profiles) can be well-approximated in the weak norm by sums of delta functions.

\begin{lemma}[Approximation by delta functions]\label{l:appx}
	For every $N_{1}\in \mathbb{N}$, $C_{1}\in (0,\infty)$, and $\epsilon\in (0,1)$ there exist $\ell_{1},C_{2}\in(1,\infty)$ such that
	for all $N\leq N_{1}$, $C\leq C_{1}$ and $\ell\geq \ell_1$ the following holds.
	Consider $2N$ points
	\begin{align}
		-\frac{\Lambda}{2}<x_1< y_1< x_2< y_2<\ldots <x_N< y_N\leq \frac{\Lambda}{2}\label{pts}
	\end{align}
	such that
	\begin{align}
		(y_i-x_i)\leq \frac{\ell}{C_2}\qquad\text{and}\qquad m_{i+1}-m_i\geq \ell\label{xy}
	\end{align}
	for all $i$, for
$
		m_i:=(x_i+y_{i-1})/2.
	$
	Consider a smooth and periodic function $f$ such that
	\begin{enumerate}
		\item[(i)]	
			for all $i\in\{1,\ldots, N\}$,
			\begin{align}\label{alphsize2}
				f(x)=\alpha_i\,w_i(x)\quad\text{on }(m_i,m_{i+1}),\qquad\text{with}\qquad
				\abs{y_i-x_i}&\leq C  \max_{i=1,\ldots,N}\abs{\alpha_{i}},
			\end{align}
			and
			\begin{align}
				\int_{m_i}^{m_{i+1}}w_{i}(x) \,\dd x&=1,\label{normedwi}\\
	    		\abs{w_{i}(x)}&\leq C\exp\left(-\frac{\dd(x,\{x_i,y_i\})}{C}\right)\quad\text{on }(m_{i},m_{i+1})\setminus (x_i,y_i).\label{expdecayf}
			\end{align}
		\item [(ii)]
			\begin{align}
				\int_{m_i}^{m_{i+1}}f^2(x)\,\dd x\leq C \left(\abs{x_i-y_i}+1\right).\label{fl2too}
			\end{align}
	\end{enumerate}

	Then for $z_i:=\frac{x_i+y_i}{2}$ we have that
	\begin{align}
		\norm*{f- \sum_{i=1}^{N}\alpha_i\delta_{z_i}}_{\N}\leq \epsilon
		\norm*{\sum_{i=1}^{N}\alpha_i\delta_{z_i}}_{\N}.\label{repbd}
	\end{align}
	Moreover, it follows that
	\begin{align}
		\norm*{f}_{\N} \sim \norm*{\sum_{i=1}^{N}\alpha_i\delta_{z_i}}_{\N} \sim \ell^\frac{1}{2}\max_{i=1,\ldots,N}\abs{\alpha_i}.\label{alphbd}
	\end{align}
\end{lemma}
With the approximation lemma in hand, it is straightforward to prove lemma \ref{l:vdist}.
\begin{proof}[Proof of lemma \ref{l:vdist}]
	Inequality \eqref{v.2} follows trivially from the Fourier representation and
	\begin{align*}
		\left(\ell^{-2}+\frac{\abs{k}^2}{\Lambda^{2}}\right)^{-1}\leq \min\left\{\ell^2,\,\frac{\Lambda^{2}}{\abs{k}^2}\right\}.
	\end{align*}
	To show \eqref{v.1}, we will use lemma \ref{l:appx} with $f:=v-\tilde{v}$,
 the points $x_i$, $y_i$ the zeros of $v$ and $\tilde{v}$, concatenated such that they are ordered as in \eqref{pts}, and
 \begin{align*}
  \alpha_i:=\int_{m_i}^{m_{i+1}}f\,\dd x.
\end{align*}
As for the conditions of lemma \ref{l:appx}, we note that \eqref{xy} holds because of the assumptions on $c$ and $\tilde c$.
The estimate \eqref{alphsize2}
 can be shown via an adaptation of the proof of \cite[Equation (3.2)]{OR} and the remainder of property (i) is straightforward.
Finally, \eqref{fl2too} is (3.4) in \cite{OR}.
	Hence, the function $f$ satisfies the bound \eqref{alphbd},
	which together with \eqref{alphsize2} implies \eqref{v.1}.
\end{proof}
It remains only to prove the approximation lemma.
\begin{proof}[Proof of lemma \ref{l:appx}]
	Notice that \eqref{alphsize2} and \eqref{normedwi} imply
	\begin{align}\label{alphai}
		\alpha_i=\int_{m_i}^{m_{i+1}}f\,\dd x.
	\end{align}
	We will use duality in the form
	\begin{align*}
		\norm{g}_{\N}=\sup_{\xi\in\N'}\frac{\int g\xi\,\dd x}{\left(\int \ell^{-2}\xi^2+\xi_x^2\,\dd x\right)^{\frac{1}{2}}} =\sup_{\norm{\xi}_{\N'}\leq 1}\int g\xi\,\dd x,
	\end{align*}
	where $\N'$ consists of periodic functions $\xi$ such that $(\ell^{-2}-\partial_{xx})^{\frac{1}{2}}\xi\in L^2$ and
	\begin{align*}
		\norm{\xi}_{\N'}^{2}=\int\xi(\ell^{-2}-\partial_{xx})\xi\,\dd x=\ell^{-2}\int\xi^2\,\dd x+\int \xi_x^2\,\dd x.
	\end{align*}
	We begin by using the first dual representation to derive a bound on $\alpha_i$. We choose for any $i=1,\ldots,N$ a test function $\xi$ such that $\xi(z_i)=\sgn(\alpha_i)$, $\abs{\xi}\leq 1$, $\abs{\xi_x}\lesssim  \ell^{-1}$,
	and $\xi\equiv 0$ on $(m_i,m_{i+1})^c$. It follows that
	\begin{align*}
		\int \ell^{-2}\xi^2+\xi_x^2\,\dd x\lesssim \frac{1}{\ell},
	\end{align*}
	so that
	\begin{align*}
		\norm*{\sum_{k=1}^{N}\alpha_k\delta_{z_k}}_{\N}\geq
		\frac{\sum_{k=1}^{N}\alpha_k\xi(z_k)}{\left(\int \ell^{-2}\xi^2+\xi_x^2\,\dd x\right)^{\frac{1}{2}}}
		\gtrsim \ell^\frac{1}{2} \abs{\alpha_i},
	\end{align*}
	which, since $i$ was arbitrary, implies
 \begin{align}
     \max_{i=1,\ldots,N}\abs{\alpha_i}\lesssim\frac{1}{\ell^\frac{1}{2}} \norm*{\sum_{i=1}^{N}\alpha_i\delta_{z_i}}_{\N}.\label{alphbdold}
 \end{align}

	We turn now to establishing \eqref{repbd}.
	Notice that for any $\xi\in\N'$ with $\norm{\xi}_{\N'}\leq 1$, one has the H\"older estimate
	\begin{align}
		\abs{\xi(y)-\xi(x)}^2=\left(\int_x^y \xi_x\,\dd s\right)^2\leq \abs{y-x}\int \xi_x^2\,\dd x\leq \abs{y-x}.\label{xihold}
	\end{align}
	We use the second dual representation and estimate for any $\xi$ with $\norm{\xi}_{\N'}\leq 1$ the pairing
	\begin{align}
		\int \left( f- \sum_{i=1}^{N}\alpha_i\delta_{z_i}\right)\xi\,\dd x\overset{\eqref{alphai}}=\sum_{i=1}^N \int_{m_i}^{m_{i+1}}f(x)\left(\xi(x)-\xi(z_i)\right)\,\dd x.\label{xisum}
	\end{align}
	Each term in the sum can be estimated via
	\begin{eqnarray}
		\lefteqn{\int_{m_i}^{m_{i+1}}f(x)\left(\xi(x)-\xi(z_i)\right)\,\dd x}\notag\\
		&\leq& \int_{(m_i,m_{i+1})\setminus(x_i,y_i)}f(x)\left(\xi(x)-\xi(z_i)\right)\,\dd x+ \int_{(x_i,y_i)}f(x)\left(\xi(x)-\xi(z_i)\right)\,\dd x\notag\\
		&\overset{\eqref{xihold}}\leq &  \int_{(m_i,m_{i+1})\setminus(x_i,y_i)}\abs{f(x)}\sqrt{\abs{x-z_i}}\,\dd x +
		\left(\int_{(x_i,y_i)}f^2(x)\,\dd x\int_{(x_i,y_i)}\abs{x-z_i}\,\dd x\right)^{\frac{1}{2}}\notag\\
		&\leq &\left(\epsilon_{2}\ell\right)^{\frac{1}{2}}\max_{i=1,\ldots,N}\abs{\alpha_i},
		\label{termest}
	\end{eqnarray}
for any $\epsilon_2>0$ by using \eqref{xy}--\eqref{fl2too} and choosing $\ell_{1}$ sufficiently large with respect to $C_1$ and $\epsilon_{2}^{-1}$. Substituting \eqref{termest} for each $i$ into \eqref{xisum} gives
	\begin{eqnarray*}
		\sup_{\norm{\xi}_{\N'}\leq 1}  \int \left( f- \sum_{i=1}^{N}\alpha_i\delta_{z_i}\right)\xi\,\dd x&\leq&  \epsilon_{2}^{\frac{1}{2}}\ell^\frac{1}{2}\,N_{1} \,\max_{i=1,\ldots,N}\abs{\alpha_i}.
	\end{eqnarray*}
	Inserting \eqref{alphbdold} and choosing $\epsilon_{2}$ sufficiently small completes the proof of \eqref{repbd}.


On the one hand, the triangle inequality
\begin{align*}
		\norm*{\sum_{i=1}^{N}\alpha_i\delta_{z_i}}_{\N}&\leq
		\norm*{f- \sum_{i=1}^{N}\alpha_i\delta_{z_i}}_{\N}+ \norm{f}_{\N}
	\end{align*}
and \eqref{repbd} with $\epsilon=\frac{1}{2}$ implies
	\begin{align}
\norm{f}_{\N}\gtrsim	 \norm*{\sum_{i=1}^{N}\alpha_i\delta_{z_i}}_{\N}\overset{\eqref{alphbdold}}\gtrsim	 \ell^\frac{1}{2}     \max_{i=1,\ldots,N}\abs{\alpha_i},\label{this}
	\end{align}
which establishes one direction in \eqref{alphbd}. On the other hand, the triangle inequality
\begin{align*}
		\norm{f}_{\N}\leq \norm*{\sum_{i=1}^{N}\alpha_i\delta_{z_i}}_{\N}+
		\norm*{f- \sum_{i=1}^{N}\alpha_i\delta_{z_i}}_{\N}
	\end{align*}
and \eqref{repbd} with $\epsilon=\frac{1}{2}$ yield
\begin{align}\label{estimatef}
  \norm{f}_{\N}\lesssim \norm*{\sum_{i=1}^{N}\alpha_i\delta_{z_i}}_{\N},
\end{align}
so that it suffices for the upper bound in \eqref{alphbd} to control the right-hand side of \eqref{estimatef} in terms of $\ell^\frac{1}{2}\max_i\abs{\alpha_i}$. For this, we turn again to the second dual representation
and estimate for any $\xi$ with $\norm{\xi}_{\N'}\leq 1$ that
\begin{align*}
  \int \xi \sum_{i=1}^{N}\alpha_i\delta_{z_i}\,\dd x=\sum_{i=1}^{N}\alpha_i\xi(z_i)\lesssim \ell^\frac{1}{2}\sum_{i=1}^{N}\abs{\alpha_i},
\end{align*}
where, according to the definition of $\norm{\cdot}_{\N'}$, we have estimated
\begin{align*}
	\MoveEqLeft
	\norm{\xi}_{L^\infty}\overset{\eqref{elementa}}\lesssim \left(\norm{\xi}_{L^2}\norm{\xi_x}_{L^2}\right)^{\frac{1}{2}}+\inf\abs{\xi}
	\leq\ell^\frac{1}{2}+\Lambda^{-1}\norm{\xi}_{L^{1}}
	\leq \ell^{\frac{1}{2}}+\Lambda^{-\frac{1}{2}}\norm{\xi}_{L^{2}}\lesssim \ell^{\frac{1}{2}}.
\end{align*}
\end{proof}

\subsection{Proofs of main lemmas}\label{ss:mainproofs}
We defer the proofs of the energy gap and energy dissipation estimates, which are lengthy, to subsections \ref{ss:app_en} und \ref{ss:app_diss}, below. In this subsection, we establish the remaining algebraic and differential estimates for the relaxation framework and the main theorem for the metastability framework.
\begin{proof}[Proof of lemma \ref{L1b}]
	First, we recall that according to lemma \ref{l:uv} we have
	\begin{align}\label{consequenceslemma1}
		\abs{c-\bar{c}}\lesssim \bar{H}^\frac{1}{3}+\bar{H}^\frac{1}{5}\qquad\text{and}\qquad
		N\lesssim E.
	\end{align}
According to \eqref{consequenceslemma1}, we may choose $\ell_1$ so large that
	\begin{align}
	\abs{c-\bar{c}}+1\ll\ell_1\leq \ell.\label{ccontrol}
	\end{align}
This control on the differences of the zeros also gives us control on the distances between zeros in the sense that
\begin{align}
  \ell(v)\geq \frac{\ell}{2}\qquad\text{and}\qquad \abs*{\min\{x_{i+1},\bar{x}_{i+1}\}-\max\{x_i,\bar{x}_i\}}\geq\frac{\ell}{2},\label{lvx}
\end{align}
where we have used $\ell(v)$ to denote the minimal distance between zeros of $v$.

We turn to the proof of \eqref{L1b.3}. Throughout the proof, we allow $\lesssim$ to include dependence on $C_H$ and $C_E$. It suffices (cf.\ notation~\ref{not:cdist}) to show that $\abs{x_{i}-\bar{x}_{i}}$ for an arbitrary $i$ is bounded by the right-hand side of \eqref{L1b.3}. Moreover we may assume without loss of generality that
$$
0=\bar{x}_i\leq x_i,\qquad\text{so that it suffices to estimate $x_i^{2}$.}
$$	
For a lengthscale $L\leq \frac{\ell}{5}$ to be specified below, we choose a positive cutoff function $\eta_L$ such that
	\begin{align*}
		\eta_L\equiv 1\text{ on }(-L,x_{i}+L),\quad \eta_L\equiv 0\text{ on }\R\setminus (-2L,x_{i}+2L),\quad\text{and}\quad
		\abs{\eta_{Lx}}\lesssim L^{-1},
	\end{align*}
and we observe (again recalling \eqref{consequenceslemma1}), that without loss
\begin{align*}
  \tilde{L}:=2L+x_i\leq \frac{\ell}{2}.
\end{align*}
It is convenient to use comparison to the infinite line minimizers.
\begin{notation}
  We denote by $v_\infty$ the function satisfying
  \begin{align*}
    -v_{\infty xx}+G'(v_{\infty})=0\quad\text{on }\R, \quad\lim_{x\to\pm \infty}v_{\infty}(x)=\pm 1,\quad v_{\infty}(0)=0,
  \end{align*}
  which is an energy minimizer subject to the $\pm 1$ boundary conditions at infinity. For the canonical potential, $v_\infty(x)=\tanh\left(\frac{x}{\sqrt{2}}\right).$
\end{notation}
We use the fact that
\begin{align*}
2x_i=  \int_\R v_\infty(x)-v_\infty(x-x_i)\,\dd x&\leq 2 \int_{-L}^{x_i+L}v_\infty(x)-v_\infty(x-x_i)\,\dd x,\\
  \int_{-L}^{x_i+L}\abs{v(x)-v_\infty(x-x_i)}\,\dd x&\leq \exp\left(-C\ell\right),\\
  \int_{-L}^{x_i+L}\abs{\bar{v}(x)-v_\infty(x)}\,\dd x&\leq \exp\left(-C\ell\right),
\end{align*}
for
\begin{align}
L\gg 1.\label{Lbig}
\end{align}
It follows that
\begin{align}
\abs{x_i}&\leq  \int \eta_L \abs{v-\bar{v}}\,\dd x+\exp\left(-C\ell\right)\notag\\
&\leq \int \eta_L \abs{v-u}\,\dd x+\int \eta_L \abs{\bar{v}-u}\,\dd x+\exp\left(-C\ell\right),\;\text{for }L\gg 1.\label{z.1}
\end{align}
For the first term on the right-hand side, we use \eqref{L1b.1} to estimate
\begin{align}
   \int \eta_L \abs{v-u}\,\dd x\lesssim
   \left(\tilde{L}\E\right)^\frac{1}{2}= \bigl((L+\abs{x_i})\E\bigr)^\frac{1}{2}.\label{z.2}
\end{align}
For the second term on the right-hand side, we estimate
\begin{align}
   \int \eta_L \abs{\bar{v}-u}\,\dd x\lesssim \left(\frac{\bar{H}}{\tilde{L}}\right)^\frac{1}{2}=\left(\frac{\bar{H}}{L+\abs{x_i}}\right)^\frac{1}{2}.\label{z.3}
\end{align}
The combination of \eqref{z.1} to \eqref{z.3} gives
\begin{align}
\abs{x_i}\lesssim \left(\frac{\bar{H}}{L+\abs{x_i}}\right)^\frac{1}{2}+ \bigl((L+\abs{x_i})\E\bigr)^\frac{1}{2} +\exp\left(-C\ell\right) \quad\text{for $L\gg 1$.} \label{z.4}
\end{align}
We now consider three cases:
\begin{align*}
  \text{(i)}\, 1\ll\left(\frac{\bar{H}}{\E}\right)^\frac{1}{2}\ll \ell,\quad \text{(ii)}\, \left(\frac{\bar{H}}{\E}\right)^\frac{1}{2}\gtrsim\ell,\quad
  \text{(iii)}\, \left(\frac{\bar{H}}{\E}\right)^\frac{1}{2}\lesssim 1.
\end{align*}
In the first case, we choose $L\sim \left(\frac{\bar{H}}{\E}\right)^\frac{1}{2}$, so that \eqref{z.4} gives
\begin{align*}
  \abs{x_i}^2\lesssim \left(\bar{H}\E\right)^\frac{1}{2}+|x_i|\E+\exp(-C\ell).
\end{align*}
On the other hand, in case (ii),
we
deduce from \eqref{z.1}-\eqref{z.3} for $L\sim \ell$ (and hence $(L+\abs{x_i})\E\lesssim \ell\E\lesssim \frac{\bar{H}}{\ell}$) that
\begin{align*}
  \abs{x_i}^2\lesssim \frac{\bar{H}}{\ell}+\exp\left(-C\ell\right).
\end{align*}
Finally, in case (iii), we use
\begin{align*}
  \frac{\bar{H}}{L}\lesssim \E\qquad\text{for $L\gtrsim 1$}.
\end{align*}
Choosing $L$ large enough so that \eqref{z.4} holds leads to
\begin{align*}
  \abs{x_i}^2\lesssim (1+\abs{x_i})\E+\exp(-C\ell).
\end{align*}

	We turn to the proof of \eqref{L1b.4}. We claim that it suffices to show
	\begin{equation}\label{l.9}
		 \E\lesssim
(\bar{H}D)^\frac{1}{2}+\left((\abs{c-\bar{c}}+1)\left(c-\bar{c}\right)^2D\right)^\frac{1}{2}+D+\exp(-C\ell).
	\end{equation}
	Indeed, substituting \eqref{L1b.3} into \eqref{l.9} gives
	\begin{align*}
		\MoveEqLeft
		\E\lesssim
(\bar{H}D)^\frac{1}{2}+\bigl((\abs{c-\bar{c}}+1)^{2}\E D\bigr)^\frac{1}{2}+\left(\frac{(|c-\bar{c}|+1) \bar{H}D}{\ell}\right)^{\frac{1}{2}}\\
		&\quad+\bigl((|c-\bar{c}|+1)\exp(-C\ell) D\bigr)^\frac{1}{2}+D+\exp(-C\ell)\\
		 \overset{\eqref{ccontrol}}\lesssim&(\bar{H}D)^\frac{1}{2}+\E^\frac{1}{2}\bigl((|c-\bar{c}|+1)^2D\bigr)^\frac{1}{2}+D+\exp(-C\ell),
	\end{align*}
	so that \eqref{L1b.4} follows from Young's inequality.
	As far as \eqref{l.9}, we note that, according to \eqref{L1b.1}, \eqref{L1b.2}, and Young's inequality, it is enough to show
	\begin{align}\label{estimateL2f}
		\begin{split}
			\MoveEqLeft
			\int {\ff}^{2}\,\dd x
			\lesssim\left(\int \F^{2}\,\dd x\int {\ff}_{x}^{2}\,\dd x \right)^{\frac{1}{2}}\\
			&\qquad \qquad+\left(\bigl((\abs{c-\bar{c}}+1)(c-\bar{c})^{2}+\exp(-C\ell)\bigr)\int {\ff}_{x}^{2}\,\dd x\right)^{\frac{1}{2}},
		\end{split}
	\end{align}
where $\F$ is defined via $\F_x=\f$, $\norm{\f}_{\dot{H}^{-1}}=\norm{\F}_{L^2}$.
	To establish \eqref{estimateL2f}, we start by estimating
	\begin{align*}
		\MoveEqLeft
		\int {\ff}^{2}\,\dd x=\int \f {\ff}\,\dd x+\int (\bar{v}-v){\ff}\,\dd x\\
		&= -\int \F{\ff}_{x}\,\dd x+\int (\bar{v}-v){\ff}\,\dd x\\
		&\leq \left(\int \F^{2}\,\dd x\int {\ff}_{x}^{2}\,\dd x\right)^{\frac{1}{2}}+\int \abs{\bar{v}-v}\abs{{\ff}}\,\dd x.
	\end{align*}
	Hence, it is sufficient for \eqref{estimateL2f} to verify
	\begin{align}\label{ineqvbarvf}
		\begin{split}
			\MoveEqLeft
			\int \abs{\bar{v}-v}\abs{{\ff}}\,\dd x\\
			&\lesssim \left(\Bigl((\abs{\bar{c}-c}+1)(\bar{c}-c)^{2}+\exp(-C\ell)\Bigr)\int {\ff}_{x}^{2}\,\dd x\right)^{\frac{1}{2}}.
		\end{split}
	\end{align}
	It is enough to show this inequality on each of the intervals $(y_{i-1},y_{i})$, where $y_{i}=\frac{x_{i}+x_{i+1}}{2}$. We argue locally on subintervals and  deduce \eqref{ineqvbarvf} by summing over these local estimates, so that the constant in \eqref{ineqvbarvf} acquires a dependence on $N$ and hence, estimating $N$ via lemma \ref{l:uv}, a dependence on $C_E$.
 For the estimate on a subinterval, we apply the Cauchy-Schwarz inequality and Hardy's inequality to estimate
	\begin{align*}
		\MoveEqLeft
		\int_{y_{i-1}}^{y_{i}}\abs{\bar{v}-v}\abs{{\ff}}\,\dd x\\
		&\leq \left(\int_{y_{i-1}}^{y_{i}}\left((x-x_{i})^{2}+1\right)(\bar{v}-v)^{2}\,\dd x \int_{y_{i-1}}^{y_{i}}\frac{{\ff}^{2}}{(x-x_{i})^{2}+1}\,\dd x\right)^{\frac{1}{2}}\\
		&\lesssim \left(\int_{y_{i-1}}^{y_{i}}\left((x-x_{i})^{2}+1\right)(\bar{v}-v)^{2}\,\dd x \int_{y_{i-1}}^{y_{i}}{\ff}_{x}^{2}\,\dd x\right)^{\frac{1}{2}},
	\end{align*}
	so that the proof is finished if we can show
	\begin{align}\label{intermed_important_ineq}
		\int_{y_{i-1}}^{y_{i}}\left((x-x_{i})^{2}+1\right)(\bar{v}-v)^{2}\,\dd x\lesssim (\abs{\bar{c}-c}+1)(\bar{c}-c)^{2}+\exp(-C\ell).
	\end{align}
	To establish this fact, we use
	\begin{align*}
		\int_{\R}\left((x-x_{i})^{2}+1\right)\bigl(v_{\infty}(x-\bar{x}_{i})-v_{\infty}(x-\bar x_{i})\bigr)^{2}\,\dd x
		\lesssim (\abs{\bar{c}-c}+1)(\bar{c}-c)^{2},
	\end{align*}
which was established in \cite[(2.12)]{OW}, together with the triangle inequality and
	\begin{align*}
		 \int_{y_{i-1}}^{y_{i}}\left((x-x_{i})^{2}+1\right)\bigl(v_{\infty}(x-x_{i})-v\bigr)^{2}\,\dd x &\lesssim \exp(-C\ell)\\
		 \int_{y_{i-1}}^{y_{i}}\left((x-x_{i})^{2}+1\right)\bigl(\bar{v}-v_{\infty}(x-\bar{x}_{i})\bigr)^{2}\,\dd x&\lesssim \exp(-C\ell),
	\end{align*}
which is not hard to show.

	Finally \eqref{L1b.5} is a consequence of the Lipschitz condition from lemma \ref{l:lip} and \eqref{lvx}.
\end{proof}

\begin{proof}[Proof of lemma \ref{L2b}]
	Equation \eqref{L2b.1} is well-known and easy to check.
	Throughout the proof, we allow $\lesssim$ to include dependence on $C_H$ and $C_E$.
	For \eqref{L2b.2}, we choose $\ell_1$ sufficiently large so that the zeros and distances between the zeros are controlled as in the proof of lemma \ref{L1b} (cf.\ \eqref{ccontrol}, \eqref{lvx}). We follow the proof of \cite[equation (1.23)]{OW} and note that in our setting
	\begin{align}\label{ft}
		\begin{split}
			\MoveEqLeft
			\f_{t}=u_{t}=\left(-u_{xx}+G'(u)\right)_{xx}
			=\left(-u_{xx}+G'(v+{\ff})\right)_{xx}\\
			 &=\left(-u_{xx}+G'(v+{\ff})+v_{xx}-G'(v)-\sum_{i=1}^{N}\alpha_{i}\delta_{x_{i}}\right)_{xx}\\
			 &=\left(G'(v+{\ff})-G'(v)-{\ff}_{xx}-\sum_{i=1}^{N}\alpha_{i}\delta_{x_{i}}\right)_{xx}
		\end{split}
	\end{align}
	in the sense of distributions, where
	\begin{align*}
		\alpha_{i}=v_x(x_{i}^{+})-v_x(x_{i}^{-}).
	\end{align*}
	From the properties of the energy optimal profiles, we read off
	\begin{align}\label{boundonalpha}
		\abs{\alpha_{i}}\lesssim \exp(-C\ell).
	\end{align}
	On the level of $\F$, equation \eqref{ft} can be written as
	\begin{align}\label{Ft}
		 \F_{t}=\left(G'(v+{\ff})-G'(v)-{\ff}_{xx}-\sum_{i=1}^{N}\alpha_{i}\delta_{x_{i}}\right)_{x}.
	\end{align}
We hence obtain for the derivative of $\bar{H}$ that
	\begin{eqnarray}\label{Hdot}
			\frac{1}{2}\frac{\dd}{\dd t} \bar{H}&=&
			\int \F\F_{t}\,\dd x\notag\\
			&\overset{\eqref{Ft}}=&\int \F\left(G'(v+{\ff})-G'(v)-{\ff}_{xx}-\sum_{i=1}^{N}\alpha_{i}\delta_{x_{i}}\right)_{x}\,\dd x\notag\\
			&=&-\int \f\left(G'(v+{\ff})-G'(v)-{\ff}_{xx}-\sum_{i=1}^{N}\alpha_{i}\delta_{x_{i}}\right)\,\dd x\notag\\
			&=&-\int ({\ff}+v-\bar{v})\left(G'(v+{\ff})-G'(v)-{\ff}_{xx}\right)\,\dd x\notag\\
			&&\qquad-\int \f\sum_{i=1}^{N}\alpha_{i}\delta_{x_{i}}\,\dd x.
	\end{eqnarray}
For the second term on the right-hand side, we calculate
	\begin{align*}
		\MoveEqLeft
		\abs*{\int \f\sum_{i=1}^{N}\alpha_{i}\delta_{x_{i}}\,\dd x}
		=\sum_{i=1}^{N}\abs{\alpha_{i}}\abs*{u(x_{i})-\bar{v}(x_{i})}
		=\sum_{i=1}^{N}\abs*{\alpha_{i}\bar{v}(x_{i})}
		\overset{\eqref{boundonalpha}}\lesssim \exp(-C\ell),
	\end{align*}	
	where we have applied the bound on $N$ from lemma \ref{l:uv}.
	For the first term on the right-hand side of \eqref{Hdot}, we write
	\begin{align}
			\MoveEqLeft
			-\int ({\ff}+v-\bar{v})\left(G'(v+{\ff})-G'(v)-{\ff}_{xx}\right)\,\dd x\notag\\
			&=-\int {\ff}_{x}^{2}+G''(v){\ff}^{2}\,\dd x-\int \left(G'(v+{\ff})-G'(v)-G''(v){\ff}\right){\ff} \,\dd x\notag\\
			&\qquad\quad -\int (v-\bar{v})\bigl(G'(v+{\ff})-G'(v)\bigr)\,\dd x-\int (v-\bar{v})_{x}{\ff}_{x}\,\dd x\notag\\
&\leq -\int \left(G'(v+{\ff})-G'(v)-G''(v){\ff}\right){\ff} \,\dd x\notag\\
		&\qquad\quad +\int (\bar{v}-v)\bigl(G'(v+{\ff})-G'(v)\bigr)\,\dd x+\int (\bar{v}-v)_{x}{\ff}_{x}\,\dd x,\label{A}
	\end{align}
	where in the last line we have used that $u$ and $v$ have common zeros, so that the linearized energy gap is positive.
	To estimate the right-hand side, we will use the elementary inequality
	\begin{align}
		\sup \abs{{\ff}}\lesssim \left( \int {\ff}^{2}\,\dd x\int {\ff}_{x}^{2}\,\dd x\right)^{\frac{1}{4}},\label{elementa}
	\end{align}
	which together with \eqref{L1b.1} and \eqref{L1b.2} implies
	\begin{align}\label{supestimatefc}
		\sup\abs{{\ff}}\lesssim \E^{\frac{1}{2}} \lesssim 1\qquad\text{and}\qquad
		\sup\abs{{\ff}}\lesssim (\E D)^{\frac{1}{4}}.
	\end{align}
	For the first integral on the right-hand side of \eqref{A}, we use \eqref{L1b.1} together with \eqref{supestimatefc} to obtain
	\begin{align*}
		\MoveEqLeft
		-\int \left(G'(v+{\ff})-G'(v)-G''(v){\ff}\right){\ff} \,\dd x\\
		&\lesssim \int \abs{{\ff}}^{3}\,\dd x\leq \sup\abs{{\ff}}\int {\ff}^{2}\,\dd x
		\lesssim
\E^{\frac{5}{4}}D^{\frac{1}{4}}\lesssim\E^{\frac{3}{4}}D^{\frac{1}{4}}.
	\end{align*}
	Using \eqref{L1b.1}, \eqref{L1b.2}, \eqref{ineqvbarvf}, and the first item in \eqref{supestimatefc}, the second integral on the right-hand side of \eqref{A} can be estimated
	\begin{align*}
		\MoveEqLeft
		\int (\bar{v}-v)\bigl(G'(v+{\ff})-G'(v)\bigr)\,\dd x
		\lesssim \int \abs{\bar{v}-v}\abs{{\ff}}\,\dd x\\
		&\lesssim_{\depEH} \left(\Bigl((\abs{\bar{c}-c}+1)\abs{\bar{c}-c}^{2}\Bigr)D\right)^{\frac{1}{2}}+\exp(-C\ell).
	\end{align*}
	The last integral on the right-hand side of \eqref{A} satisfies
	\begin{align*}
		\int (\bar{v}-v)_{x}{\ff}_{x}\,\dd x
		\leq \left(\int (\bar{v}_{x}-v_x)^{2}\,\dd x\int {\ff}_{x}^{2}\,\dd x\right)^{\frac{1}{2}}
	\end{align*}
	and, similarly to in \eqref{intermed_important_ineq}, we estimate
	\begin{align*}
		\left(\int (\bar{v}_{x}-v_x)^{2}\,\dd x\right)^{\frac{1}{2}}
		\lesssim \min\{1,\abs{c-\bar{c}}\}+\exp(-C\ell).
	\end{align*}
	Combining these two estimates with \eqref{L1b.1} and \eqref{L1b.2}, we deduce
	\begin{align*}
		\int (\bar{v}-v)_{x}{\ff}_{x}\,\dd x
		\lesssim
\left((\abs{\bar{c}-c}+1)\abs{\bar{c}-c}^{2}D\right)^{\frac{1}{2}}+\exp(-C\ell).
	\end{align*}

\end{proof}

\subsection{Scaling of the energy gap}\label{ss:app_en}

In this section we prove \eqref{L1b.1} of lemma \ref{l:eedraw}. While in \cite{OR} we assumed for convenience that $u$ had a sign on large subintervals, we make no such assumption here. We collect the necessary ingredients, show how they combine to establish \eqref{L1b.1}, and then give the proofs of the lemmas.
We begin by recalling the $L^\infty$ bound of Modica and Mortola in $d=1$.
\begin{lemma}\label{l:linf}
  For all $u:\left(-\frac{\Lambda}{2},\frac{\Lambda}{2}\right]\to\R$ with at least one zero, there holds
  \begin{align*}
    \sup\abs{u}\lesssim 1+E(u).
  \end{align*}
\end{lemma}
Because we will work locally, it is convenient to introduce the following notation.
\begin{notation}
  In the rest of this section, we assume unless otherwise noted that:
 \begin{itemize}
   \item $u$ is a function such that $u(0)=u(\ell)=0$,
   \item $v$ is the positive energy-optimal profile subject to $v(0)=v(\ell)=0$,
   \item and integrals, norms, and energy are over $(0,\ell)$.
 \end{itemize}
\end{notation}We recall the estimate of the linearized energy gap from \cite{OR}.
\begin{lemma}[Linearized energy gap estimate,~\cite{OR} proposition~3.1]\label{l:line}
	There exists $\ell_1,\,{C_{\mathrm{ed}}}\in (0,\infty)$ such that for all $\ell\geq \ell_1$, for all
	smooth functions $f$ with $f(0)=f(\ell)=0$, there holds
	\begin{align}
		\norm{f}_{H^1}^2\leq \frac{C_{\mathrm{ed}}}{4}\,\int
		f_x^2+G''(v)f^2\,\dd x.\label{lin.en}
	\end{align}
\end{lemma}

As in \cite{OR}, we seek to improve from the linear estimate to a nonlinear estimate, using that functions with small energy gap are close to an energy-optimal profile.
To establish this, we adapt the proof from \cite[lemma 3.6]{OR} to our setting.

\begin{lemma}[Small energy gap implies uniform closeness]\label{uniformcloseness}
	There exists $\ell_2\in(0,\infty)$ with the following property.	
	For every $\epsilon>0$, there exists $\gamma>0$ such that for all $\ell\geq \ell_{2}$, there holds
	\begin{align*}
		\energy(u)-\energy(v)\leq \gamma\qquad\Rightarrow\qquad
		\min_{\mp}\sup_{[0,\ell]}\abs{u\mp v}\leq \epsilon.
	\end{align*}
\end{lemma}
The preceding lemmas lead to the following nonlinear energy gap inequality in the case of small energy gap.

\begin{lemma}[Energy gap inequality for small energy gap]\label{eginequsmall}
	There exist $\ell_{2},C_{\mathrm{ed}},\gamma\in(0,\infty)$, such that for all $\ell\geq\ell_{2}$, there holds
	\begin{align*}
		\energy(u)-\energy(v)\leq \gamma\qquad\Rightarrow\qquad
		\min_{\mp}\norm{u\mp v}_{H_{1}}^{2}\leq C_{\mathrm{ed}}\big(\energy(u)-\energy(v)\big).
	\end{align*}
\end{lemma}


On the other hand, we use a rough bound in the case of large energy gap.
\begin{lemma}[Energy gap inequality for large energy gap]\label{egilarge}
	For every $C_{1}\in (0,\infty)$, $\gamma>0$, $\ell\geq\ell_{2}$ there exists $C_{\mathrm{ed}}\in (0,\infty)$ such that for $\norm{u-v}_{L^{2}}^2\leq C_{1}$, there holds
	\begin{align*}
		\energy(u)-\energy(v)\geq \gamma\qquad\Rightarrow\qquad
		\norm{u-v}_{H_{1}}^{2}\leq C_{\mathrm{ed}}\big(\energy(u)-\energy(v)\big).
	\end{align*}
\end{lemma}

\begin{proof}[Proof of \eqref{L1b.1} of lemma \ref{l:eedraw}]
From lemma \ref{l:linf} we deduce the right-hand side inequality of \eqref{L1b.1}.
Indeed, for $u\in \mathcal{M}_N(\ell,C_H,C_E)$  we estimate
\begin{eqnarray}
	\energy(u)-\energy(v)
	&=&\int\frac{1}{2}(u_{x}-v_x)^{2}+v_x(u_{x}-v_x)+G(u)-G(v)\,\dd x\notag\\
&=&\int\frac{1}{2}(u_{x}-v_x)^{2}-v_{xx}(u-v)+G(u)-G(v)\,\dd x\notag\\
	 &\stackrel{\eqref{stationary_AC}}{=}&\int\frac{1}{2}(u_{x}-v_x)^{2}-G'(v)(u-v)+G(u)-G(v)\,\dd x\notag\\
	&{\leqsim}& \norm{u_{x}-v_x}_{L^{2}}^{2}+\sup_{\abs{s}\leq 1+\norm{u}_{L^\infty}}\abs{G''(s)}\norm{u-v}_{L^2}^{2}
	\leqsim_{C_{E}} \norm{u-v}_{H^{1}}^{2}.\label{eeda}
\end{eqnarray}
The left-hand inequality in \eqref{L1b.1} follows from the combination of lemma \ref{l:uv}, the bound \eqref{eeda}, lemma \ref{eginequsmall}, and lemma \ref{egilarge}. Both estimates hold uniformly for energy gap smaller than the constant $\gamma$ from lemma \ref{eginequsmall}.
\end{proof}

\begin{proof}[Proof of lemma \ref{uniformcloseness}]
	Now suppose for a contradiction that there are interval lengths $\ell_{n}\to\ell\in [\ell_{2},\infty]$ and functions $u_{n}\in H_{0}^{1}((0,\ell_{n}))$
	such that
	\begin{align}\label{consequenceofassumptions}
		\energy_{\ell_{n}}(u_{n})-\energy_{\ell_{n}}(v_{\ell_{n}})\to 0\quad\text{while}\quad
		\lim_{n\to\infty}\sup_{[0,\ell_{n}]}\min_{\mp}\abs{u_{n}\mp v_{\ell_{n}}}\geq \epsilon>0,
	\end{align}
	where $v_{\ell_n}$, $E_{\ell_n}$ denote the positive energy-optimal profile and energy on $(0,\ell_n)$.

	We observe the following uniform estimate for any $b\in (0,\ell)$ and $n$ large enough:
	\begin{align}
	 	\norm{u_{n}}_{H^1((0,b))}^{2}&\leq \bigl(1+b^{2}\bigr)\norm{u_{nx}}_{L^2((0,b))}^{2}\lesssim (1+b^{2})\energy_{\ell_{n}}(u_{n})\overset{\eqref{consequenceofassumptions}}\lesssim 1+b^{2},\label{unib}
	\end{align}
	where we have used the Poincar\'e inequality and $v_\ell$ is the uniform limit on $(0,b)$ of $v_{\ell_n}$.
	Thus, we find a subsequence and a limit $u\in H^1$ such that on compact subsets we have that
	\begin{align}\label{unlocalconvergence}
		u_{n}\rightharpoonup u\quad\text{in }H^1,\quad
		u_{n}\rightarrow u\quad\text{in }C\text{ and }L^{2}.
	\end{align}

	\step{Step 1.}
		We show $u=\pm v_{\ell}$.
		In case $\ell<\infty$, this follows by $b\uparrow \ell$ in
		\begin{align*}
			\energy_{b}(u)\stackrel{\text{\eqref{unlocalconvergence}}}{\leq} \liminf_{n\to\infty}\energy_{b}(u_{n})
			\leq \liminf_{n\to\infty}\energy_{\ell_{n}}(u_{n})
			 \stackrel{\text{\eqref{consequenceofassumptions}}}{=}\liminf_{n\to\infty}\energy_{\ell_{n}}(v_{\ell_{n}})
			=\energy_{\ell}(v_{\ell}),
		\end{align*}
		together with the uniqueness (up to the sign) of the energy-optimal profile.

		In case $\ell=\infty$, we cut $u_{n}$ and $v_{\ell_{n}}$ at $\frac{\ell_{n}}{2}$, reflect the right-hand side about the horizontal axis and glue this function
		at the origin to the remaining part, i.e., we consider the mapping
		\begin{align*}
			\tilde f:\left(-\tfrac{\ell_{n}}{2},\tfrac{\ell_{n}}{2}\right)\to\R,\, x\mapsto
			\begin{cases}
				-f(\ell_{n}-x),&x\in \bigl(-\frac{\ell_{n}}{2},0\bigl)\\
				f(x),&x\in \bigl[0,\frac{\ell_{n}}{2}\bigr).
			\end{cases}
		\end{align*}
		The energy is invariant under this transformation.
		The  argument from above gives \eqref{unlocalconvergence} on compact subsets for $\tilde u_{n}, \tilde u$, and $\tilde v_{\infty}$.
		On the other hand, we estimate the energy via
		\begin{eqnarray*}
			\energy_{(-b,b)}(\tilde u)&\stackrel{\text{\eqref{unlocalconvergence}}}{\leq} & \liminf_{n\to\infty}\energy_{(-b,b)}(\tilde u_{n})
			\leq \liminf_{n\to\infty}\energy_{(-\frac{\ell_{n}}{2},\frac{\ell_{n}}{2})}(\tilde u_{n})\\ &\stackrel{\text{\eqref{consequenceofassumptions}}}{=}&\liminf_{n\to\infty}\energy_{(-\frac{\ell_{n}}{2},\frac{\ell_{n}}{2})}(\tilde v_{\ell_{n}})
			=\energy_{(-\infty,\infty)}(v_{\infty})=c_{0}.
		\end{eqnarray*}
		Taking $b\uparrow \infty$, we see that the energy of $\tilde u$ is bounded by that of a kink. Finite energy gives $\abs{\tilde u(\pm\infty)}=1$.
		Recalling $\tilde u(0)=0$ and considering the uniform energy bound on $\tilde{u}_n$ and the method of construction, we infer $\tilde{u}=\pm v_\infty$.\\

	\step{Step 2.}
		In this step we will derive a contradiction.
		In case $\ell<\infty$, the uniform convergence from \eqref{unlocalconvergence} and the uniform convergence of $v_{\ell_n}$ to $v_\ell$ rule out the second item in \eqref{consequenceofassumptions}.
		For the case $\ell=\infty$, we will use the uniform convergence of $\tilde u_{n}$ to $\pm v_{\infty}$ on compact sets from step 1 to reach a contradiction.
		Indeed, we deduce:
		\begin{align}
			&\tilde{u}_n(\pm b) \text{ is arbitrarily close to $\pm 1$ or $\mp 1$ for $b$ and $n$ large enough},\label{abc}\\
			&\min_{\mp}\norm{\tilde{u}_{n}\mp \tilde{v}_{\ell_{n}}}_{L^{\infty}((-b,b))}< \frac{\epsilon}{2} \text{ for any $b$ and for $n$ large enough},\label{unifromconvergenceb}\\
			& E_{(-b,b)}(\tilde{u}_n)\geq c_0-\delta\text{ for any $\delta>0$, for $b$ and $n$ large enough},\label{okdk}
		\end{align}
		where in the last item we have used \eqref{abc} and the trick of Modica and Mortola. From \eqref{unifromconvergenceb} and
		\eqref{consequenceofassumptions}, we deduce the existence of a point $x_n\in(b,\frac{\ell_n}{2})\cup (-\frac{\ell_n}{2},-b)$ such that
		\begin{align}
			\min_{\mp}\abs{\tilde{u}_n(x_n)\mp\tilde{v}_{\ell_n}(x_n)}\geq \epsilon.\notag
		\end{align}
		Without loss of generality, suppose that $x_n\in(b,\frac{\ell_n}{2})$ so that for $b$ large there holds in addition
		\begin{align}
			\min_{\mp}\abs{\tilde{u}_n(x_n)\mp 1}\geq \frac{\epsilon}{2}.  \label{bigtrouble}
		\end{align}
		Combining \eqref{abc} and \eqref{bigtrouble} and again applying the Modica-Mortola trick, we deduce the existence of $\delta>0$ such that
		\begin{align*}
			E_{(b,\frac{\ell_n}{2})}(\tilde{u}_n)\geq 2\delta.
		\end{align*}
		But together with \eqref{okdk}, this contradicts  the first item in \eqref{consequenceofassumptions} via
		\begin{align*}
			E_{(-\frac{\ell_n}{2},\frac{\ell_n}{2})}(\tilde{u}_n)=E_{\ell_n}(u_n)\geq c_0+\delta \qquad \text{for $\delta>0$ and all }n.
		\end{align*}
\end{proof}
\begin{proof}[Proof of lemma \ref{eginequsmall}]
	With lemmas \ref{l:linf}, \ref{l:line}, and \ref{uniformcloseness} in hand, the result follows directly from Taylor's formula (cf. \cite[lemma 3.7]{OR}).
\end{proof}
\begin{proof}[Proof of lemma \ref{egilarge}]
	We use the simplistic estimate
	\begin{align*}
		\norm{u-v}_{H^{1}}^{2}&\leq 2\norm{u-v}_{L^2}^2+4\bigl(\energy(u)+\energy(v)\bigr)
		\leq 2C_{1}+8E(v)+4\bigl(\energy(u)-\energy(v)\bigr)\\
		&\leq \left(\frac{C}{\gamma}+4\right)\bigl(\energy(u)-\energy(v)\bigr).
	\end{align*}
\end{proof}

\subsection{Dissipation estimate}\label{ss:app_diss}

Here we prove \eqref{L1b.2} of lemma \ref{l:eedraw}. As in the previous subsection, we collect the ingredients, show how they combine to establish \eqref{L1b.2}, and finish with the proofs of the lemmas.
We begin with a variant of \cite[Lemma 3.4]{OW} for the half space $\R_{+}$.

\begin{lemma}[Kernel of linear dissipation operator]\label{kernellineardisspationhalfspace}
	If $f\in C^{3}(\R_{+})\cap H^{1}(\R_{+})$ satisfies
	\begin{align}\label{operatorLequalzerohalfspace}
		\operatorL f=\bigl(-f_{xx}+G''(v_{\infty})f\bigr)_{x}=0,
	\end{align}
	then $f=\alpha v_{\infty x}$ for some $\alpha\in\R$.
\end{lemma}
Next we deduce a linearized dissipation estimate (similar to \cite[Lemma 3.2]{OW}).
\begin{lemma}[Linearized dissipation estimate]\label{linearizeddiss}
	For every $C_{1}$ there exist $\ell_{1},C_{2}\in (0,\infty)$ such that for all $\ell\geq \ell_{1}$ and all smooth functions $f$ with
	\begin{align*}
		f(0)=f(\ell)=0,\qquad \text{and}\qquad \int_{0}^{\ell}f^{2}\,\dd x\leq C_{1},
	\end{align*}
	there holds
	\begin{align}\label{linearizeddissipationestimate}
		\int_{0}^{\ell}\frac{f^{2}}{1+x^{2}}+f_{x}^{2}+f_{xx}^{2}+f_{xxx}^{2}\,\dd x
		\leq C_{2}\int_{0}^{\ell}\left(\bigl(f_{xx}-G''(v_{\ell})f\bigr)_{x}\right)^{2}\,\dd x.
	\end{align}
\end{lemma}

The next goal is to show that small dissipation yields an $L^\infty$ bound.
\begin{lemma}[Small dissipation implies uniform closeness]\label{smalldissipationuniformclose}
	For every $\epsilon>0$ and every $C_{E},C_{1}>0$, there is an $\ell_{*}>0$ and $\gamma>0$ such that for all $\ell\in (\ell_{*},\infty)$ and
	\begin{align}\label{hypnonlindis}
		u\in H_{0}^{3}\bigl((0,\ell)\bigr)\qquad\text{with}\qquad
		\energy(u)\leq C_{E}\qquad\text{and}\qquad
		\int_{0}^{\ell}(u-v_{\ell})^{2}\,\dd x\leq C_{1},
	\end{align}
	there holds
	\begin{align}\label{smalldissipationlinfty}
		\dissipation(u)\leq \gamma\qquad\Rightarrow\qquad
		\sup_{[0,\ell]}\abs{u-v_{\ell}}\leq \epsilon.
	\end{align}
\end{lemma}
These tools suffice for the nonlinear dissipation estimate \eqref{L1b.2}.
\begin{proof}[Proof of \eqref{L1b.2} of lemma \ref{l:eedraw}]
	We consider a subinterval $I=(x_i,x_{i+1})$ between two zeros.
	By $D_I$ we denote the restriction of the dissipation to an interval:
	\begin{align*}
		D_I:=\int_I \left(\big(-u_{xx}+G'(u)\big)_x\right)^2\,\dd x.
	\end{align*}
	It suffices to show
	\begin{align*}
		\int_I {\ff}_{x}^2\,\dd x\lesssim D_I.
	\end{align*}
	Trivially, $D_I\gtrsim 1$ and $E(u)\leq C_E$ imply
	\begin{align*}
		\int_I {\ff}_{x}^2 \,\dd x\lesssim 1+C_E\lesssim_{\depE} D_I.
	\end{align*}
	Hence, it suffices to show that there exist $\gamma\in(0,\infty)$ such that, for
	\begin{align}
		D\leq \gamma,\label{dsmall}
	\end{align}
	there holds $\int_I {\ff}_{x}^2 \,\dd x\lesssim D_I.$ A useful identity in this part is
	\begin{align}
		\MoveEqLeft
		-u_{xx}+G'(u)=-{\ff}_{xx}+G'(u)-G'(v)\notag\\
		&=\big(-{\ff}_{xx}+G''(v){\ff}\big)+\big(G'(u)-G'(v)-G''(v){\ff}\big).\label{iddis}
	\end{align}
	We will use that small dissipation results in small $\norm{{\ff}}_{L^\infty(I)}$ (see lemma \ref{smalldissipationuniformclose}), together with the linearized dissipation estimate from lemma \ref{linearizeddiss}.
	Note that the energy of $u$ is bounded by assumption, so that \eqref{L1b.1} ensures the $L^{2}$ bound on $f$ that is necessary to apply lemmas \ref{linearizeddiss} and \ref{smalldissipationuniformclose}.
	Using \eqref{iddis} together with the triangle inequality, the linear dissipation estimate, and Taylor's formula leads to the lower bound
	\begin{align*}
		D_I\geq \frac{1}{2{C_{\mathrm{ed}}}}\int_I \frac{{\ff}^{2}}{1+x^{2}}+{\ff}_{x}^{2}\,\dd x-C\norm{{\ff}}_{L^\infty(I)}\int_I \frac{{\ff}^{2}}{1+x^{2}}+{\ff}_{x}^{2}\,\dd x,
	\end{align*}
	where ${C_{\mathrm{ed}}}$ is the constant from the linearized estimate \eqref{linearizeddissipationestimate} (see equation (3.34) in \cite{OW}).
	We now use lemma \ref{smalldissipationuniformclose} to choose $\gamma>0$ such that \eqref{dsmall} implies
	\begin{align*}
		\norm{{\ff}}_{L^\infty(I)}\leq \frac{1}{4 {C_{\mathrm{ed}}}C}.
	\end{align*}
\end{proof}
\begin{proof}[Proof of lemma \ref{kernellineardisspationhalfspace}]
	We first note that from the regularity of $f$ and \eqref{operatorLequalzerohalfspace} it follows that $f\in H^{3}(\R_{+})$ and hence
	\begin{align}\label{btoinfty}
		\lim_{x\to\infty}f(x)=\lim_{x\to\infty}f_{x}(x)=\lim_{x\to\infty}f_{xx}(x)=0.
	\end{align}
	Integrating \eqref{operatorLequalzerohalfspace} from $x$ to $\infty$ gives
	\begin{align}\label{lambdaequalzerohalfspace}
		-f_{xx}+G''(v_{\infty})f=0.
	\end{align}
	According to \eqref{lambdaequalzerohalfspace} and the equation for $v$, we have
	\begin{align*}
		(f_{x}v_{\infty x}-fv_{\infty xx})_{x}=f_{xx}v_{\infty x}-fv_{\infty xxx}
		=G''(v_{\infty})fv_{\infty x}-fG''(v_{\infty})v_{\infty x}=0,
	\end{align*}
	so that
	\begin{align*}
		f_{x}v_{\infty x}-fv_{\infty xx}=\lambda\qquad\text{for some}\quad\lambda\in\R.
	\end{align*}
	From \eqref{btoinfty} and the properties of $v_{\infty}$, we deduce $\lambda=0$. Hence $f=\alpha v_{\infty x}$ follows from
	\begin{align*}
		\frac{\dd}{\dd x}\left(\frac{f}{v_{\infty x}}\right)=\frac{f_{x}v_{\infty x}-fv_{\infty xx}}{v_{\infty x}^{2}}=0.
	\end{align*}
\end{proof}

\begin{proof}[Proof of lemma \ref{linearizeddiss}]
	\step{Step 1.}
		We start by showing that it is sufficient to establish
		\begin{align}\label{keyest}
			\int_{0}^{\ell}f_{x}^{2}\,\dd x\lesssim  \int_{0}^{\ell}\Bigl(\bigl(f_{xx}-G''(v_{\ell})f\bigr)_{x}\Bigr)^{2}\,\dd x.
		\end{align}
		Indeed, the first summand in \eqref{linearizeddissipationestimate} is bounded by Hardy's inequality together with \eqref{keyest}.
		By interpolation, it suffices to bound the third derivatives.
		On the one hand, we obtain
		\begin{align*}
			\int_{0}^{\ell}f_{xxx}^{2}\,\dd x
			\lesssim  \int_{0}^{\ell}\Bigl(f_{xxx}-\bigl(G''(v_{\ell})f\bigr)_{x}\Bigr)^{2}\,\dd x+\int_{0}^{\ell}\Bigl(\bigl(G''(v_{\ell})f\bigr)_{x}\Bigr)^{2}\,\dd x.
		\end{align*}
		On the other hand, we have that
		\begin{align*}
			\MoveEqLeft
			\int_{0}^{\ell}\Bigl(\bigl(G''(v_{\ell})f\bigr)_{x}\Bigr)^{2}\,\dd x
			\lesssim \int_{0}^{\ell}\Bigl(G'''(v_{\ell})v_{\ell x}f\Bigr)^{2}\,\dd x+\int_{0}^{\ell}\bigl(G''(v_{\ell})f_{x}\bigr)^{2}\,\dd x\\
			&\lesssim \sup_{(0,\ell)}\abs*{\left(1+x^{2}\right)\left(G'''(v_{\ell})v_{\ell x}\right)^{2}}\int_{0}^{\ell}\frac{f^{2}}{1+x^{2}}\,\dd x
			+ \int_{0}^{\ell}f_{x}^{2}\,\dd x,
		\end{align*}
		which is bounded by the Hardy inequality and the properties of $v_{\ell}$.\\

	\step{Step 2.}
		Suppose to the contrary that there is a sequence $\ell_{n}\to\infty$ and a corresponding
		sequence of functions $f_{n}:[0,\ell_{n}]\to\R$ and a constant $C_{1}>1$ such that
		\begin{align}\label{propertiesoffn}
			\int_{0}^{\ell_{n}}f_{nx}^{2}\,\dd x=1,\quad\int_{0}^{\ell_{n}}f_{n}^{2}\,\dd x\leq C_{1}\quad
			f_{n}(0)=0,\quad\text{and}\quad
			\int_{0}^{\ell_{n}}(\operatorL_{n} f_{n})^{2}\,\dd x\to 0.
		\end{align}
		Here, we set
		\begin{align*}
			\operatorL_{n}f_{n}:=\bigl(-f_{nxx}+G''(v_{\ell_{n}}) f_{n}\bigr)_{x}
		\end{align*}
		As in step 1, we obtain the improved uniform bound
		\begin{align}\label{uniformboundfn}
			 \int_{0}^{\ell_{n}}\frac{f_{n}^{2}}{1+x^{2}}+f_{nx}^{2}+f_{nxx}^{2}+f_{nxxx}^{2}\,\dd x\lesssim 1.
		\end{align}

	\step{Step 3.}
		We claim that there is a subsequence and a limit function $f$ such that on compact sets we have that
		\begin{align}
			\begin{split}\label{propertiessubsequence}
			f_{n}\rightharpoonup f\text{ in } H^{3},\quad
			f_{n}\to f\text{ in } C^{2},\quad f(0)=0,\quad
			\text{and}\quad
			\operatorL f=0\quad\text{on}\quad\R_+,
			\end{split}
		\end{align}
		for $\operatorL f:=\bigl(-f_{xx}+G''(v_{\infty}) f\bigr)_{x}$.
		The weak convergence follows from the uniform bound \eqref{uniformboundfn} and the strong convergence is a consequence of the
		compact embedding $H^{3}\Subset C^{2}$. 
		The strong convergence yields $f(0)=0$.
		For $\operatorL f=0$ we use \eqref{propertiesoffn} and weak lower semi-continuity.
		Hence $f\in C^2(\R_+)$ improves to $f\in C^{3}(\R_{+})$.
		A similar argument using the first and second item in \eqref{propertiesoffn} yields $f\in H^{1}(\R_{+})$. Hence, according to lemma \ref{kernellineardisspationhalfspace}, we conclude that
		$
			f=\alpha v_{x}
		$
		for some $\alpha\in\R$. The condition $f(0)=0$ then forces $f\equiv 0$.

		Note that by passing to a subsequence, we can additionally guarantee that \eqref{propertiessubsequence} holds with $f_{n}$ replaced by $\tilde f_{n}(x):=f_{n}(\ell_{n}-x)$.\\

	\step{Step 4.}
		We will now use the $L^{2}$ control on $\operatorL f$ and $f_{x}$ to contradict the first item in \eqref{propertiesoffn}.
		As a starting point we use the local $C^{2}$ convergence to zero to pass from $\operatorL f$ to a spatially constant operator in the sense that we show
		\begin{align}\label{convergencespatiallyconstant}
			\int_{0}^{\ell_{n}}\Bigl(\bigl(f_{nxx}-G''(1)f_{n}\bigr)_{x}\Bigr)^{2}\,\dd x\to 0.
		\end{align}
		According to the triangle inequality
		\begin{align*}
			\MoveEqLeft
			 \abs*{\left(\int_{0}^{\ell_{n}}\Bigl(\bigl(f_{nxx}-G''(v_{\ell_{n}})f_{n}\bigr)_{x}\Bigr)^{2}\,\dd x\right)^{\frac{1}{2}}
			-\left(\int_{0}^{\ell_{n}}\Bigl(\bigl(f_{nxx}-G''(1)f_{n}\bigr)_{x}\Bigr)^{2}\,\dd x\right)^{\frac{1}{2}}}^{2}\\
			&\leq \int_{0}^{\ell_{n}}\biggl(\Bigl(\bigl(G''(v_{\ell_{n}})-G''(1)\bigr)f_{n}\Bigr)_{x}\biggr)^{2}\,\dd x,
		\end{align*}
		it suffices to show that the right-hand side converges to zero.
		From the local $C^{2}$ convergence of of $f_{n}$ and $\tilde f_{n}$, we see that it is enough to show that, for any $\epsilon>0$, there exists an $X<\infty$ such that
		\begin{align*}
			 \limsup_{n\to\infty}\int_{X}^{\ell_{n}-X}\biggl(\Bigl(\bigl(G''(v_{\ell_{n}})-G''(1)\bigr)f_{n}\Bigr)_{x}\biggr)^{2}\,\dd x\leq\epsilon.
		\end{align*}
		This can be seen by using the uniform bound \eqref{uniformboundfn} and the fact that
		\begin{align*}
			\sup_{x\in [X,\ell_{n}-X]}\abs{G''(v_{\ell_{n}})-G''(1)}^{2}\qquad\text{and}\qquad
			\sup_{x\in [X,\ell_{n}-X]}\left(1+x^{2}\right)\abs{G'''(v_{\ell_{n}})v_{\ell_{n},x}}^{2}
		\end{align*}
		can be made small by choosing $X$ large.\\

	\step{Step 5.}
		We note that
		\begin{align*}
			\MoveEqLeft
			\int_{0}^{\ell_{n}}\Bigl(\bigl(f_{nxx}-G''(1)f_{n}\bigr)_{x}\Bigr)^{2}\,\dd x
			 =\int_{0}^{\ell_{n}}\left(f_{nxxx}^{2}-2G''(1)f_{nx}f_{nxxx}+G''(1)^{2}f_{nx}^{2}\right)\,\dd x\\
			 &=\int_{0}^{\ell_{n}}\left(f_{nxxx}^{2}+2G''(1)f_{nxx}^{2}+G''(1)^{2}f_{nx}^{2}\right)\,\dd x-\left[2G''(1)f_{nx}f_{nxx}\right]_{0}^{\ell_{n}},
		\end{align*}
		and hence
		\begin{align*}
			G''(1)^{2}\int_{0}^{\ell_{n}}f_{nx}^{2}\,\dd x\leq
			\int_{0}^{\ell_{n}}\Bigl(\bigl(f_{nxx}-G''(1)f_{n}\bigr)_{x}\Bigr)^{2}\,\dd x+\left[2G''(1)f_{nx}f_{nxx}\right]_{0}^{\ell_{n}}
			\to 0
		\end{align*}
		by \eqref{convergencespatiallyconstant} and local $C^{2}$ convergence of $f_{n}$, $\tilde f_{n}$ to $0$.
		This contradicts the first item in \eqref{propertiesoffn}.
\end{proof}
\begin{proof}[Proof of lemma \ref{smalldissipationuniformclose}]
	\step{Step 1.}
		Assume to the contrary that there are interval lengths $\ell_{n}\to\infty$ and functions $u_{n}\in H_{0}^{1}((0,\ell_{n}))$ such that
		\begin{align}\label{consequenceofassumptionsdissipation}
			\energy(u_{n})\leq C_{E},\quad
			\norm{u_{n}-v_{n}}_{L^{2}}\leq C_{1},\quad
			\dissipation(u_{n})\to 0,\;\text{ and }\;
			\lim_{n\to\infty}\sup_{[0,\ell_{n}]}\abs{u_{n}-v_{\ell_{n}}}\geq \epsilon>0.
		\end{align}
		According to the bound on the energy and lemma \ref{l:linf}, we have that
		\begin{align}\label{energyubounded}
			\sup_{[0,\ell_{n}]}\abs{u_{n}}\leq C\qquad\text{and}\qquad \int_{0}^{\ell_{n}}u_{nx}^{2}\,\dd x\leq C_{E}.
		\end{align}
		Additionally, we have
		\begin{align*}
			\MoveEqLeft
			\biggl(\int_{0}^{\ell_{n}}u_{nxxx}^{2}\,\dd x\biggr)^{\frac{1}{2}}
			\leq \biggl(\int_{0}^{\ell_{n}}\bigl(u_{nxx}-G'(u)\bigr)_{x}^{2}\,\dd x\biggr)^{\frac{1}{2}}+\biggl(\int_{0}^{\ell_{n}}\big(G''(u_{n})u_{nx})^{2}\,\dd x\biggr)^{\frac{1}{2}}\\
			&\leq \bigl(\dissipation(u_{n})\bigr)^{\frac{1}{2}}+\sup_{[-C,C]}\abs{G''}\bigl(\energy(u_{n})\bigr)^{\frac{1}{2}}\leq C^{\frac{1}{2}}+\sup_{[-C,C]}\abs{G''}C_{E}^{\frac{1}{2}}.
		\end{align*}
		By interpolation
		\begin{align*}
			 \sup_{[0,\ell_{n}]}\abs{u_{n}}+\int_{0}^{\ell_{n}}u_{nx}^{2}+u_{nxx}^{2}+u_{nxxx}^{2}\,\dd x\leq C,
		\end{align*}
		so that
		we find a subsequence such that
		\begin{align}\label{unlocalconvergencedissipation}
			u_{n}\rightharpoonup u\quad\text{in }H_{\mathrm{loc}}^{3}\qquad\text{and}\qquad
			u_{n}\rightarrow u\quad\text{in }C_{\mathrm{loc}}^{2}.
		\end{align}
		As in the proof of the linearized dissipation estimate, we obtain
		\begin{align}\label{uvL2close2}
			\norm{u-v_{\infty}}_{L^{2}((0,\infty))}\leq C_{1},\qquad
			\energy(u)\leq C_{E},\qquad\text{and}\qquad\dissipation(u)=0.
		\end{align}

	\step{Step 2.}
		We will show that $u$ is a equal to $v_{\infty}$.
		The last item in \eqref{uvL2close2} implies
		\begin{align}\label{eqc1}
			u_{xx}-G'(u)=\lambda.
		\end{align}
		From $E(u)\leq C_{E}$ and the first item in \eqref{uvL2close2}, we infer $u\to +1$ for $x\to\infty$.
		Using this information in \eqref{eqc1} yields
		\begin{align*}
			\lim_{x\to\infty}u_{xx}(x)=\lambda.
		\end{align*}
		Boundedness of the energy of $u$ gives $\lambda=0$.
		Since the only bounded energy solutions to \eqref{eqc1} on $\R_{+}$ with $\lambda=0$ and $u(0)=0$ are $\pm v_{\infty}$,
		we get $u=v_{\infty}$ from the first item in \eqref{uvL2close2}.\\

	\step{Step 3.}
		We will use the second and fourth conditions in \eqref{consequenceofassumptionsdissipation} to arrive at a contradiction.
		According to step 1, $u_{n}$ is uniformly close to a hyperbolic tangent  on compact intervals around $0$ and $\ell_{n}$.
		By the second and fourth conditions in \eqref{consequenceofassumptionsdissipation} we know that both these hyperbolic tangent profiles must be positive, so that $u_{n}$ is also uniformly close to $v_{\ell_{n}}$ on these intervals.
		Passing to a subsequence if necessary, this means that the last condition in \eqref{consequenceofassumptionsdissipation} implies the existence of $x_{n}$ with
		\begin{align*}
			 x_{n}\to\infty,\qquad
			 \ell_{n}-x_{n}\to\infty,\qquad
			 \abs{u_{n}(x_{n})-1}\geq \tfrac{\epsilon}{2},\qquad\text{and}\qquad
			 u_{nx}(x_{n})=0.
		\end{align*}
		Repeating the argument of the first step for $\tilde u_{n}(x)=u_{n}(x_{n}+x)$
		we find $\tilde u$ such that
		\begin{align*}
			\tilde u_{n}\rightharpoonup \tilde u\quad\text{in }H_{\mathrm{loc}}^{3}\qquad\text{and}\qquad
			\tilde u_{n}\rightarrow \tilde u\quad\text{in }C_{\mathrm{loc}}^{2}
		\end{align*}
		as well as
		\begin{align}\label{conditions_tilde_u}
			\energy(\tilde u)\leq C_{E},\qquad D(\tilde u)=0,\qquad
			\abs{\tilde u(0)-1}\geq \tfrac{\epsilon}{2},\qquad\text{and}\qquad
			\tilde u_{x}(0)=0.
		\end{align}
		Additionally, the last item in \eqref{hypnonlindis} gives
		\begin{align}\label{tildeu1}
			\norm{\tilde u-1}_{L^{2}}\leq C_{1}.
		\end{align}
		As in the previous step, we deduce that $\tilde u$ satisfies \eqref{eqc1} with $\lambda=0$.
		Since the only bounded energy solutions to \eqref{eqc1} on $\R_{+}$ with $\lambda=0$ and $u_{x}(0)=0$ are $\pm 1$, the third item in \eqref{conditions_tilde_u} implies $\tilde u\equiv -1$.
		This contradicts \eqref{tildeu1}.
\end{proof}

\subsection{Dissipation proofs from section \ref{ss:post}}\label{ss:disspf}
To show lemma \ref{l:diss1}, we will use comparison with energy minimizers on $\R$ and the following linearized energy gap estimate.
\begin{lemma}\label{l:oyh}
  For all $f\in H^1(\R)$  such that
\begin{align*}
 \int_{\R}
f\,v_{\infty x}\,\dd x=0, \quad\text{there holds}\quad \int_{\R} f^2+f_x^2\,\dd x\lesssim\int_{\R}
f_x^2+G''(v_\infty)f^2\,\dd x.
\end{align*}
\end{lemma}
This lemma is easy to prove by combining Lemma 3.1 in \cite{OW} with a bootstrapping argument as in the proof of proposition 3.1 in \cite{OR}.
The proof of lemma \ref{l:diss1} now follows by adapting the proof of the differential inequality for $D$ from \cite{OW}.
\begin{proof}[Proof of lemma \ref{l:diss1}.]
A straightforward calculation and periodic boundary conditions give
\begin{align*}
  \frac{\dd}{\dd t}\frac{1}{2}D=-\int h_x^2+G''(u)h^2\,\dd x,
\end{align*}
where $h:=(u_{xx}-G'(u))_{xx}$, so that it suffices to show
\begin{align*}
  \int h_x^2+G''(u)h^2\,\dd x\gtrsim -D.
\end{align*}
The idea is to use the linearized energy gap estimate on a large interval around zeros of $v$. It suffices to give the details around one zero; without loss of generality consider $x_i=0$ and recall that the  closest neighboring element of $c$ is at least $\ell\geq\ell_1$ away.

We introduce a partition of unity $\eta_-$, $\eta $, $\eta_+$ such that
\begin{align}
  &\eta \equiv 1\quad\text{on}\quad \left[\frac{x_{i-1}}{4},\frac{x_{i+1}}{4}\right],\label{eta.1}\\
  &\sppt(\eta )\Subset\left(\frac{3x_{i-1}}{4},\frac{3x_{i+1}}{4}\right),\notag\\
  &\abs{\eta_{xx}}\leq \frac{2}{\ell^2},\label{eta.3}
\end{align}
while setting
$
  \eta_-=(1-\eta )\textbf{1}_{x\leq 0},\;\eta_+=(1-\eta )\textbf{1}_{x\geq 0}.
$
We claim that
\begin{align}
  \int h_x^2\,\dd x&\geq \int \bigl((\eta_- h)_x\bigr)^2+ \bigl((\eta  h)_x\bigr)^2+ \bigl((\eta_+ h)_x\bigr)^2\,\dd x\notag\\
  &\qquad-\frac{2}{\ell^2}\int (\eta_-+\eta+\eta_+)h^2\,\dd x,\label{sec.1}\\
  \int G''(u)h^2\,\dd x&\geq\int G''(u)\left((\eta_- h)^2+(\eta  h)^2+(\eta_+ h)^2\right)\,\dd x.\label{sec.2}
\end{align}
Indeed, the inequality of the integrands in both cases is clear whenever two of the partition functions vanish and it suffices to check
the ``overlap regions'' $I_-:=[\frac{3x_{i-1}}{4},\frac{x_{i-1}}{4}]$ and $I_+:=[\frac{x_{i+1}}{4},\frac{3x_{i+1}}{4}]$. We choose
$\ell_1$ sufficiently large and $\eps$ sufficiently small so that
\begin{align}
G''\bigl(u(x)\bigr)\geq \frac{1}{ C_1}>0\quad\text{for }x\in I_-\cup I_+.\label{L1}
\end{align}
From this choice and
\begin{align}
\eta_-^2+\eta^2+\eta_+^2\leq 1,\label{eta123}
\end{align}
we deduce \eqref{sec.2}. For \eqref{sec.1}, we use \eqref{eta123}, \eqref{eta.1}, and \eqref{eta.3} together with the identity
\begin{align*}
  \int \big((\eta h)_x\big)^2\,\dd x&=\int \eta^2h_x^2+2\eta\eta_x h\,h_x +\eta_x^2 h^2\,\dd x=\int \eta^2h_x^2-\eta\eta_{xx}h^2\,\dd x
\end{align*}
for $\eta_-,\,\eta$, and $\eta_+$.

We have hence localized the estimates. It suffices to show
\begin{align}
  \int \Big((\eta h)_x\Big)^2 +G''(u)(\eta h)^2\,\dd x-\frac{2}{\ell^2}\int \eta h^2\,\dd x\gtrsim - D.\label{suff.1.f}
\end{align}
On the support of $\eta$, we project $\eta h$ onto $v_\infty$ (centered at zero) via
\begin{align}
  \eta h= h_0 +\alpha v_{\infty x},\quad\text{with }\alpha:=\frac{\int \eta h v_{\infty x}\,\dd x}{\int v_{\infty x}^2\,\dd x}.\label{deco}
\end{align}
Via integration by parts and the decay of $v_{\infty x},\,v_{\infty xx}$, we observe that
\begin{align}
  \alpha^2\lesssim D.\label{alphad}
\end{align}
Substituting the decomposition for $\eta h$ and using the equation for $v_\infty$ gives
\begin{align*}
  \lefteqn{\int \Big((\eta h)_x\Big)^2 +G''(u)(\eta h)^2\,\dd x}\\
  &=\int \Big((\eta h)_x\Big)^2 +G''(v_\infty)(\eta h)^2\,\dd x+\int\Big( G''(u)-G''(v_\infty)\Big)(\eta h)^2\,\dd x\\
  &=\int h_{0x}^2 +G''(v_\infty)h_0^2\,\dd x+\int\Big( G''(u)-G''(v_\infty)\Big)(\eta h)^2\,\dd x.
\end{align*}
On the one hand, $h_0\in H^1(\R)$ and is orthogonal to $v_{\infty x}$, so that we can use the positivity of the energy gap from lemma \ref{l:oyh}.
On the other hand, from the $L^\infty$ bound on $u-v$ and the convergence of $v$ to $v_\infty$ with $\ell$, we deduce
\begin{align}
  \sup_{\left[\frac{3x_{i-1}}{4},\frac{3x_{i+1}}{4}\right]}\abs{G''(u)-G''(v_\infty)}\lesssim \eps+o(1)_{\ell\to\infty}\leq 2\eps\label{linfuv}
\end{align}
for $\ell_1$ large enough. Applying these facts yields
\begin{align}
 	\int \Big((\eta h)_x\Big)^2 +G''(u)(\eta h)^2\,\dd x
	\gtrsim \int h_0^2+h_{0 x}^2\,\dd x -C\eps \int (\eta h)^2\,\dd x.\label{suff.1.f2}
\end{align}
According to \eqref{suff.1.f2}, $\abs{\eta}\leq 1$, and Young's inequality, it suffices for \eqref{suff.1.f} to show
\begin{align}
  \int \eta h^2\,\dd x\lesssim D^{1/2}\left(\int h_{0 x}^2\,\dd x\right)^{1/2}+D.\label{suff.g}
\end{align}
Turning again to the decomposition \eqref{deco} and the definition of $h$, we find
\begin{align*}
  \int \eta h^2\,\dd x=\int g_{xx}\eta h\,\dd x=-\int g_x (\eta h)_x\,\dd x=-\int g_x (h_{0 x}+\alpha v_{\infty xx})\,\dd x,
\end{align*}
so that \eqref{suff.g} follows from
the Cauchy-Schwarz inequality and \eqref{alphad}.
\end{proof}

We now show that a bound on $\eb$ induces a bound on $D$.
\begin{proof}[Proof of lemma \ref{l:diss3}]
  For the statement to be nontrivial, we assume that $t\geq s+1$. According to \eqref{L1b.5} and lemma \ref{l:eedraw}, we may assume $\ell_1$ and $\gamma$ to be such that \eqref{d1.down} holds. Let
  \begin{align*}
    t_1:={\rm argmin}_{[s,s+1]}D,\qquad t_2:={\rm argmax}_{[s+1,t]}D.
  \end{align*}
  Then on the one hand, the minimum on $[s,s+1]$ is bounded above via
   \begin{align}
   D(t_1)\leq \int_{s}^{s+1} D(\tau)\,\dd\tau\overset{\eqref{L2b.1}}{=} \int_{s}^{s+1}-\frac{\dd}{\dd t}\eb(\tau)\,\dd\tau\leq \abs*{\eb(s)}+\abs*{\eb(s+1)}\leq 2\gamma.\label{d1.1}
  \end{align}
  On the other hand, the maximum on $[s+1,t]$ is bounded via
  \begin{align*}
    \max_{[s+1,t]}D= D(t_1)+\int_{t_1}^{t_2}\frac{\dd}{\dd\tau}D(\tau)\,\dd t\overset{\eqref{d1.1},\eqref{d1.down}}
    \lesssim \gamma+\int_{t_1}^{t_2} D(\tau)\,\dd\tau\overset{\eqref{L2b.1}}{\lesssim}\gamma.
  \end{align*}
\end{proof}
\begin{proof}[Proof of lemma \ref{l:diss2}]
As in the proof of \cite[equation (1.19)]{OW}, it suffices to show that $\norm{\xi}_{L^\infty(\torus)}$ is small, where $\xi:=\frac{u_{x}^{2}}{2}- G(u)$ is the so-called discrepancy. Indeed, on a neighborhood of any zero $x_i$, $u$ satisfies
\begin{align}
  u_x=\pm \sqrt{2(G(u)+\xi)}\quad\text{with}\quad u(x_i)=0.\label{gode}
\end{align}
Smallness of the discrepancy hence forces $u$ to take on values near $\pm 1$ in between any pair of zeros, and smallness of the energy gap and the observation of Modica and Mortola then rules out more than $N$ zeros.

Fix $y\in\torus$ and let $\eta$ be a cut-off function such that $\eta\equiv 1$ on $(y-L,y+L)$ and $\sppt(\eta)\Subset (y-2L,y+2L)$.
 Letting $g:=u_{xx}-G'(u)$ we observe
\begin{align}
	\label{linftyboundg}
	\sup_{[y-L,y+L]}|g|\leq \sup\abs{\eta g}\lesssim L^{-\frac{1}{2}}+(LD)^{\frac{1}{2}}\lesssim D^{\frac{1}{4}}+\Lambda^{-\frac{1}{2}},
\end{align}
where in the last inequality we have optimized in $L$ subject to $L\lesssim\Lambda$.
Since $y$ was arbitrary, we obtain \eqref{linftyboundg} for $\abs{g}$ on all of $\torus$.
Again fixing $y\in\torus$ and using this estimate together with $\xi_{x}=gu_{x}$ and the energy bound gives
\begin{align*}
	\sup_{[y-L,y+L]}\abs{\xi}\leq \abs*{\frac{1}{L}\int_{[y-L,y+L]}\xi\,\dd x}+\int_{[y-L,y+L]}\abs{\xi_{x}}\,\dd x
	\lesssim L^{-1}+ \left(D^{\frac{1}{4}}+\Lambda^{-\frac{1}{2}}\right)L^{\frac{1}{2}}.
\end{align*}
Optimizing in the interval size and recalling that $y$ was arbitrary, we obtain
\begin{align*}
	\norm{\xi}_{L^\infty(\torus)}\lesssim D^{\frac{1}{6}}+\Lambda^{-\frac{1}{3}}.
\end{align*}
\end{proof}

\section{Metastability proofs}\label{S:meta}

In the first proposition, we establish metastability under weak norm conditions without assuming integrability of $t\mapsto E(v(t))$.
\begin{proof}[Proof of proposition \ref{prop:orweak2}]
\step{Step 1: Preliminaries}.
	We introduce the two quantities
	\begin{align*}
	\mathcal{E}(t):=E\bigl(u(t)\bigr)-E\bigl(v(t)\bigr),\qquad \mathcal{E}_s(t):=E\bigl(u(t)\bigr)-E\bigl(v(s)\bigr).
	\end{align*}
	Notice that $\mathcal{E}(t)\geq 0$. Although $\mathcal{E}_s$ is not a positive quantity, we will control the negative part via
	\begin{align}
	 \mathcal{E}_s(t)\geq -\abs*{E\bigl(v(t)\bigr)-E\bigl(v(s)\bigr)}.\label{u6}
	\end{align}
	Also notice that from the energy $E$, $\mathcal{E}_s$ inherits continuity as well as monotonicity:
	\begin{align}
	 \mathcal{E}_s(t_{2})\leq \mathcal{E}_s(t_{1})\qquad\text{for}\quad 0\leq t_{1}\leq t_{2}.\label{u1.1}
	\end{align}
Finally, as in the proposition we define
\begin{align*}
  e_0:=\mathcal{E}(0)=\mathcal{E}_0(0) \quad\text{and note that }\quad e_0\geq 0.
\end{align*}

We want to use \eqref{eed2.2} and integration in time to develop an integral equation that gives the decay of $\mathcal{E}_0$ (up to error terms). Our task in this step is to control changes along the slow manifold as measured in terms of $\mathcal{E}_0$. To this end, we use the gradient flow inequality with $w\equiv 1$ and $0\leq s<t<T$ as in \cite{OR} to bound
\begin{align}
 \norm{u(t)-u(s)}_1^{2}&\leq (t-s)\Big(E\bigl(u(s)\bigr)-E\bigl(u(t)\bigr)\Big)\label{alsou1}\\
 &\leq (t-s) \Big( E\bigl(u(s)\bigr)-E\bigl(v(s)\bigr)+E\bigl(v(s)\bigr)-E\bigl(v(t)\bigr)+E\bigl(v(t)\bigr)-E\bigl(u(t)\bigr)\Big)\notag\\
 &\leq (t-s) \Big( \mathcal{E}(s)+\abs*{E\bigl(v(s)\bigr)-E\bigl(v(t)\bigr)}\Big),\label{u1}
\end{align}
where we have dropped the negative term.
With the triangle inequality, this gives control on changes along the slow manifold as
\begin{eqnarray}
\lefteqn{  \norm{v(t)-v(s)}_0}\notag\\
 &\leq&  \norm{v(t)-u(t)}_0 +\norm{u(t)-u(s)}_0+\norm {u(s)-v(s)}_0 \label{u.tri}\\
 &\overset{\eqref{eed2.2},\eqref{u3.4},\eqref{u1}}\leq&  \sqrt{2\mathcal{E}(t)}+\sqrt{2\mathcal{E}(s)}+\biggl((t-s)\left( \mathcal{E}(s)+\abs*{E\bigl(v(t)\bigr)-E\bigl(v(s)\bigr)}\right)\biggr)^{\frac{1}{2}}. \label{u.vchg}
\end{eqnarray}
Plugging into the Lipschitz condition \eqref{lipschitz.2} gives
\begin{eqnarray}
\lefteqn{ \abs*{E\bigl(v(t)\bigr)-E\bigl(v(s)\bigr)}}\notag\\
 &\overset{\eqref{eed2.2},\eqref{u.vchg}}\leq& \delta\left(  \sqrt{2\mathcal{E}(t)}+\sqrt{2\mathcal{E}(s)}+\biggl((t-s)\left( \mathcal{E}(s)+\abs*{E\bigl(v(t)\bigr)-E\bigl(v(s)\bigr)}\right)\biggr)^{\frac{1}{2}}\right).\label{u2}
\end{eqnarray}
From the positivity of $\mathcal{E}(\cdot)$, we observe
\begin{align*}
  \mathcal{E}(t)\leq \mathcal{E}_s^+(t)+\abs*{E\bigl(v(t)\bigr)-E\bigl(v(s)\bigr)},
\end{align*}
where we have denoted the positive part by $f^+:=\max\{f,0\}$. Using this and the elementary inequality $\sqrt{A+B}\leq \sqrt{A}+\sqrt{B}$ for $A,\,B\geq 0$, we obtain from \eqref{u2} the bound
\begin{align}
  \lefteqn{ \abs*{E\bigl(v(t)\bigr)-E\bigl(v(s)\bigr)}}\notag\\
  &\leq \delta\Biggl(\sqrt{2\mathcal{E}_s^+(t)}+\sqrt{2\abs*{E\bigl(v(t)\bigr)-E\bigl(v(s)\bigr)}}+\sqrt{2\mathcal{E}(s)}+
  \sqrt{t-s}\left(\sqrt{\mathcal{E}(s)}+\sqrt{\abs*{E\bigl(v(t)\bigr)-E\bigl(v(s)\bigr)}}\right)\Biggr).\notag
\end{align}
Using $\delta(t-s)\leq\delta t\leq 1$, $\delta\leq 1$ and Young's inequality, we deduce
\begin{align}
\abs*{E\bigl(v(t)\bigr)-E\bigl(v(s)\bigr)}
\leq 2\delta\sqrt{2\mathcal{E}_s^+(t)}+ C\left(\sqrt{\delta\mathcal{E}(s)}+\delta\right).\label{u3.b}
\end{align}

\step{Step 2: Rough bound on energy decay.}
On the one hand, we derive from \eqref{u6} and \eqref{u3.b} with $s=0$ a bound on negative values of $\mathcal{E}_0$:
\begin{align}
 - \mathcal{E}_0(t)\lesssim \sqrt{\delta e_0}+\delta.\label{u1.4}
\end{align}
On the other hand, we now derive our initial integral inequality. Using \eqref{eed2.2} in the form
\begin{align}
  2\mathcal{E}_0(t)+\frac{\dd}{\dd t}E\bigl(u(t)\bigr)\leq 2\abs*{E\bigl(v(0)\bigr)-E\bigl(v(t)\bigr)}\label{original}
\end{align}
and applying \eqref{u3.b} with $s=0$, we obtain after an integration over $(t_1,t_2)$  the inequality
\begin{eqnarray}
  \lefteqn{\mathcal{E}_0(t_2)-\mathcal{E}_0(t_1)+(2-\delta)\int_{t_1}^{t_2} \mathcal{E}_0(\tau)\,\dd\tau}\notag\\
&\lesssim& \left(\sqrt{\delta e_0}+\delta\right)(t_2-t_1)-\delta\int_{t_1}^{t_2} \mathcal{E}_0(\tau)\textbf{1}_{\{\mathcal{E}_0(\tau)<0\}}\,\dd\tau\notag\\
&\overset{\eqref{u1.4}}\lesssim &\left(\sqrt{\delta e_0}+\delta\right)(t_2-t_1).\label{u1.2}
\end{eqnarray}

This is the differential inequality for the gap $e_0$ from which we will deduce exponential in time decay (up to error terms).
Indeed, let
\[F(t):=\int_{t-1}^t\mathcal{E}_0(\tau)\,\dd\tau.\]
According to \eqref{u1.1}, $F$ satisfies the initial condition
\[F(1)\leq e_0.\]
On the other hand, recalling that $\delta$ is small enough so that $2-\delta\geq 1$, we have from \eqref{u1.2} that
\begin{align*}
 F'(t)+ F(t)\lesssim \sqrt{\delta e_0}+\delta.
\end{align*}
Integrating this differential inequality yields
\begin{align*}
F(t)= \int_{t-1}^t\mathcal{E}_0(\tau)\,\dd\tau\lesssim \exp(- t)e_0+\sqrt{\delta e_0}+\delta.
\end{align*}
Recalling the monotonicity \eqref{u1.1}, we deduce
\begin{align*}
  \mathcal{E}_0(t)\lesssim \exp(-t)e_0+\sqrt{\delta e_0}+\delta.
\end{align*}
and hence, in light of the lower bound \eqref{u1.4}, we have arrived at
\begin{align}
\abs{ \mathcal{E}_0(t)}\lesssim \exp(-t)e_0+\sqrt{\delta e_0}+\delta.\label{u2.2}
\end{align}

\step{Step 3: The initial layer and better energy decay.}
The goal of this step is to identify $T_1\geq 0$ such that \eqref{13.s1} and \eqref{13.s2} hold.
Along the way we obtain an improved decay rate for the initial energy gap; cf. \eqref{u.u.dec}, below.

We separate into cases.
\uline{Case 1}: If $e_0\leq \delta^2$, then $e_0\leq \delta e_0^{\frac{1}{2}}$ and we set $T_1=0$.

\uline{Case 2}:
If $e_0> \delta^2,$ we define the initial relaxation time via
\begin{align*}
  T_1:=\ln\left( \frac{e_0^{\frac{1}{2}}}{\delta}\right).
\end{align*}
By assumption on $\delta$, there holds $T_1\leq \delta^{-1}$. On $[0,t]$ for $t\leq T_1$ we obtain from the gradient flow inequality with weight $w(s):=\exp(\frac{s}{4})$:
\begin{align*}
  \norm{u(t)-u(0)}_1^2&\lesssim E\bigl(u(0)\bigr)-\exp\left(\tfrac{t}{4}\right)E\bigl(u(t)\bigr)+\int_0^t\exp\left(\tfrac{s}{4}\right)E\bigl(u(s)\bigr)\,\dd s.
\end{align*}
We rewrite this in terms of $\mathcal{E}_0(\cdot)$ and $\mathcal{E}(\cdot)$ as:
\begin{align}
 \norm{u(t)-u(0)}_1^2&\lesssim e_0+E\bigl(v(0)\bigr)-\exp\left(\tfrac{t}{4}\right)\mathcal{E}(t)-\exp\left(\tfrac{t}{4}\right)E\bigl(v(t)\bigr)\notag\\
 &\quad +\int_0^t\exp\left(\tfrac{s}{4}\right)\mathcal{E}_0(s)\,\dd s+\int_0^t\exp\left(\tfrac{s}{4}\right)E\bigl(v(0)\bigr)\,\dd s\notag\\
 &\leq e_0+E\bigl(v(0)\bigr)-\exp\left(\tfrac{t}{4}\right)E\bigl(v(t)\bigr)\notag\\
 &\quad +\int_0^t\exp\left(\tfrac{s}{4}\right)\mathcal{E}_0(s)\,\dd s+\int_0^t\exp\left(\tfrac{s}{4}\right)E\bigl(v(0)\bigr)\,\dd s.\label{u2.1}
\end{align}
Notice that by redefining the energy (at the expense of giving up positivity), we may assume $E(v(0))=0$, and that the Lipschitz condition then gives
$
  \abs*{E\bigl(v(t)\bigr)}\leq \delta \norm{v(t)-v(0)}_0.
$
Using these facts in \eqref{u2.1}, we observe
\begin{align*}
  \norm{u(t)-u(0)}_1^2&\lesssim e_0+\delta\exp\left(\tfrac{t}{4}\right)\norm{v(t)-v(0)}_0+\int_0^t\exp\left(\tfrac{s}{4}\right)\mathcal{E}_0(s)\,\dd s.
\end{align*}
Now we use \eqref{u2.2}, $t\leq T_1$, and $\delta\lesssim {e_0^{\frac{1}{2}}}$ to estimate
\begin{align}
  \norm{u(t)-u(0)}_1^2&\lesssim {e_0}+e_0^{\frac{1}{2}}+\delta\exp\left(\tfrac{T_1}{4}\right)\norm{v(t)-v(0)}_0.\label{u2.3}
\end{align}
Using the triangle inequality as in \eqref{u.tri} together with this estimate, we find
\begin{eqnarray*}
  \norm{v(t)-v(0)}_0&\lesssim& \sqrt{\mathcal{E}(t)}+{e_0^{\frac{1}{2}}}+e_0^{\frac{1}{4}}+\biggl(\delta\exp\left(\tfrac{T_1}{4}\right)\norm{v(t)-v(0)}_0\biggr)^{\frac{1}{2}}\\
  &\lesssim& \sqrt{\abs{\mathcal{E}_0(t)}}+{e_0^{\frac{1}{2}}}+e_0^{\frac{1}{4}}+\biggl(\delta\exp\left(\tfrac{T_1}{4}\right)\norm{v(t)-v(0)}_0\biggr)^{\frac{1}{2}}\\
  &\overset{\eqref{u2.2}}\lesssim&{e_0^{\frac{1}{2}}}+e_0^{\frac{1}{4}}+\biggl(\delta\exp\left(\tfrac{T_1}{4}\right)\norm{v(t)-v(0)}_0\biggr)^{\frac{1}{2}}.
\end{eqnarray*}
From Young's inequality, we deduce
\begin{align}
  \norm{v(t)-v(0)}_0\lesssim {e_0^{\frac{1}{2}}}+e_0^{\frac{1}{4}}+\delta\exp\left(\tfrac{T_1}{4}\right)\lesssim {e_0^{\frac{1}{2}}}+e_{0}^{\frac{1}{4}}.\label{u5.1}
\end{align}
Plugging back into \eqref{u2.3} yields in addition
\begin{align*}
   \norm{u(t)-u(0)}_1^2\lesssim e_0+e_0^{\frac{1}{2}},
\end{align*}
which completes the proof of \eqref{13.s2}. We remark that \eqref{u5.1} leads to a better energy decay rate on $[0,T_1]$. Indeed, recalling \eqref{original} and using \eqref{lipschitz.2} and \eqref{u5.1}, we obtain the improved integral inequality
\begin{align}
  2\int_s^t\mathcal{E}_0(\tau)\,\dd\tau+\mathcal{E}_0(t)-\mathcal{E}_0(s)\lesssim \delta\left({e_0^{\frac{1}{2}}}+e_0^{\frac{1}{4}}\right)(t-s).\notag
\end{align}
Proceeding as in step 2 returns
\begin{align}
\abs{ \mathcal{E}_0(t)}\lesssim \exp(-2t)e_0+\delta\left({e_0^{\frac{1}{2}}}+e_0^{\frac{1}{4}}\right)\quad\text{for all }t\in[0,T_1].\label{u.u.dec}
\end{align}
The bound on $\mathcal{E}(t)$ follows from this estimate, the Lipschitz condition \eqref{lipschitz.2}, and \eqref{u5.1}.\\

\step{Step 4: Order one changes for long times.}
Now we consider the evolution for $t\in [T_1,1/\delta]$. Using the estimates from step 1 with $s=T_1$ as in step 2 gives
\begin{align}
\abs{  \mathcal{E}_{T_1}(t)}\lesssim  \delta\left({e_0^{\frac{1}{2}}}+e_0^{\frac{1}{4}}\right)\qquad\text{ for all $T_1\leq t\leq \delta^{-1}$},\label{u3.1}
\end{align}
and hence, according to \eqref{lipschitz.2} and the previous step, also
\begin{align}
 \abs{ \mathcal{E}_0(t)}\lesssim \delta\left({e_0^{\frac{1}{2}}}+e_0^{\frac{1}{4}}\right)\qquad\text{ for all $T_1\leq t\leq \delta^{-1}$}.\label{u.thishere}
\end{align}
The combination of \eqref{u.u.dec} and \eqref{u.thishere} completes the proof of \eqref{u.first}.
From here it also follows that
\begin{eqnarray}
\frac{1}{2}\norm{u(t)-v(t)}_0^2\leq  \mathcal{E}(t)\lesssim \delta\left({e_0^{\frac{1}{2}}}+e_0^{\frac{1}{4}}\right)+\delta\norm{v(t)-v(T_1)}_0
\label{u3.2}
\end{eqnarray}
for all $t\geq T_1$. Proceeding similarly to \eqref{alsou1}, we deduce
\begin{eqnarray}
 \norm{u(t)-u(T_1)}_1^2&\lesssim& \bigl(\abs{\mathcal{E}_{T_1}(T_1)}+\abs{\mathcal{E}_{T_1}(t)}\bigr)(t-T_1)\notag\\
 &\overset{\eqref{u3.1}}\lesssim& \delta\left({e_0^{\frac{1}{2}}}+e_0^{\frac{1}{4}}\right)(t-T_1).\label{u3.3}
 \end{eqnarray}
This shows \eqref{diffu}.
Using the triangle inequality as in \eqref{u.tri} (with $s=T_1$) together with \eqref{eed2.2} and this estimate yields
\begin{eqnarray*}
\norm{v(t)-v(T_1)}_0^2 &\lesssim&  \mathcal{E}(t)+\mathcal{E}(T_1)+\delta\left({e_0^{\frac{1}{2}}}+e_0^{\frac{1}{4}}\right)(t-T_1)\notag\\
 &\overset{\eqref{u3.1}}\lesssim& \delta\left({e_0^{\frac{1}{2}}}+e_0^{\frac{1}{4}}\right)(1+t-T_1)+\delta\norm{v(t)-v(T_1)}_0 .
\end{eqnarray*}
From Young's inequality and $\delta\lesssim 1$ we obtain \eqref{diffv}.
\end{proof}

In the second proposition, we verify that the abstract result from \cite{OR} holds true if one considers a weaker norm. We indicate the (small) changes in the original proof for the convenience of the reader.
\begin{proof}[Proof of proposition \ref{prop:orweak}]
	The proof of lemma 2.1 in \cite{OR} still works in the modified setting.
	To see this, we notice that according to \eqref{lipschitz.2} it is enough to estimate the weak norm via
	\begin{align*}
		\norm{v(s)-v(t)}_{0}\overset{\eqref{eed2.2}}\leq \sqrt{2e(s)}+\sqrt{2e(t)}+\norm{u(t)-u(s)}_{0}
	\end{align*}
	in order to control the energy gap on the slow manifold $\mathcal{N}$. The last summand can be estimated in the stronger norm $\norm{\cdot}_{1}$ as in \cite[(2.6)]{OR} by
	the gradient flow inequality \cite[(1.8)]{OR}. After this, the proof makes use only of clever but elementary calculations and does not employ any of the assumptions.

	For the first part, i.e. $t\geq t_{1}$, in the proof of lemma 2.2 in \cite{OR}, we use \cite[(2.5) and (2.6)]{OR} to estimate
	\begin{align*}
		\MoveEqLeft
		\norm{u(t)-u(t_{1})}_{1}^{2}\leq (t-t_{1})\left(e(t_{1})+\delta\Big(\sqrt{2e(t_{1})}+\sqrt{2e(t)}+\norm{u(t)-u(t_{1})}_{0}\Big)\right)\\
		&\overset{\mbox{\scriptsize\cite[(2.20)]{OR}}}\lesssim (t-t_{1})\delta^{2}+(t-t_{1})\delta\norm{u(t)-u(t_{1})}_{0}.
	\end{align*}
	Young's inequality gives
	\begin{align*}
			\norm{u(t)-u(t_{1})}_{1}\lesssim \delta(1+t-t_{1}).
	\end{align*}
	For the second part, everything carries through as in the original proof if one measures the difference between $u(t)$ and $u(0)$ in the strong
	norm associated to the gradient flow and the remaining quantities in the weak norm.
\end{proof}

\section*{Acknowledgements}
It is a pleasure to acknowledge valuable input from Felix Otto. In addition we gratefully acknowledge the hospitality of the Max Planck Institute for Mathematics in the Sciences and the Institut des Hautes \'Etudes Scientifiques, as well as useful discussions with Lia Bronsard and Wadim Gerner. S.\ Scholtes was
partially supported by DFG Grant WE 5760/1-1.

%

\begin{bibdiv}
\begin{biblist}


\bib{ABF}{article}{
      author={Alikakos, Nicholas},
      author={Bates, Peter~W.},
      author={Fusco, Giorgio},
       title={Slow motion for the {C}ahn-{H}illiard equation in one space
  dimension},
        date={1991},
     journal={J. Differential Equations},
      volume={90},
      number={1},
       pages={81\ndash 135},
}

\bib{num}{article}{
      author={Argentina, Mederic},
      author={Clerc, MG},
      author={Rojas, R},
      author={Tirapegui, E},
       title={Coarsening dynamics of the one-dimensional Cahn-Hilliard model},
        date={2005},
     journal={Physical Review E},
      volume={71},
      number={4},
       pages={046210},
}

\bib{BX1}{article}{
      author={Bates, Peter~W.},
      author={Xun, Jian~Ping},
       title={Metastable patterns for the {C}ahn-{H}illiard equation. {I}},
        date={1994},
     journal={J. Differential Equations},
      volume={111},
      number={2},
       pages={421\ndash 457},
}

\bib{BX2}{article}{
      author={Bates, Peter~W.},
      author={Xun, Jian~Ping},
       title={Metastable patterns for the {C}ahn-{H}illiard equation. {II}.
  {L}ayer dynamics and slow invariant manifold},
        date={1995},
     journal={J. Differential Equations},
      volume={117},
      number={1},
       pages={165\ndash 216},
}

\bib{BH}{article}{
      author={Bronsard, Lia},
      author={Hilhorst, Danielle},
       title={On the slow dynamics for the {C}ahn-{H}illiard equation in one
  space dimension},
        date={1992},
     journal={Proc. Roy. Soc. London Ser. A},
      volume={439},
      number={1907},
       pages={669\ndash 682},
}

\bib{BK}{article}{
      author={Bronsard, Lia},
      author={Kohn, Robert~V.},
       title={On the slowness of phase boundary motion in one space dimension},
        date={1990},
     journal={Comm. Pure Appl. Math.},
      volume={43},
      number={8},
       pages={983\ndash 997},
}

\bib{CP}{article}{
      author={Carr, J.},
      author={Pego, R.~L.},
       title={Metastable patterns in solutions of
  {$u_t=\epsilon^2u_{xx}-f(u)$}},
        date={1989},
     journal={Comm. Pure Appl. Math.},
      volume={42},
      number={5},
       pages={523\ndash 576},
}

\bib{C}{article}{
      author={Chen, Xinfu},
       title={Generation, propagation, and annihilation of metastable
  patterns},
        date={2004},
     journal={J. Differential Equations},
      volume={206},
      number={2},
       pages={399\ndash 437},
}

\bib{EF}{article}{
      author={Elliott, Charles~M.},
      author={French, Donald~A.},
       title={Numerical studies of the {C}ahn-{H}illiard equation for phase
  separation},
        date={1987},
     journal={IMA J. Appl. Math.},
      volume={38},
      number={2},
       pages={97\ndash 128},
}

\bib{FH}{article}{
      author={Fusco, G.},
      author={Hale, J.~K.},
       title={Slow-motion manifolds, dormant instability, and singular
  perturbations},
        date={1989},
     journal={J. Dynam. Differential Equations},
      volume={1},
      number={1},
       pages={75\ndash 94},
}

\bib{G}{article}{
      author={Grant, Christopher~P.},
       title={Slow motion in one-dimensional {C}ahn-{M}orral systems},
        date={1995},
     journal={SIAM J. Math. Anal.},
      volume={26},
      number={1},
       pages={21\ndash 34},
}

\bib{OR}{article}{
      author={Otto, Felix},
      author={Reznikoff, Maria~G.},
       title={Slow motion of gradient flows},
        date={2007},
     journal={J. Differential Equations},
      volume={237},
      number={2},
       pages={372\ndash 420},
}


\bib{OSW}{article}{
	author={Otto, Felix},
	author={Scholtes, Sebastian},
	author={Westdickenberg, Maria G.},
	journal={work in progress},
	title={{Relaxation to equilibrium in the one-dimensional Cahn--Hilliard equation, Part II}},
	year={2017},
}

\bib{OW}{article}{
      author={Otto, Felix},
      author={Westdickenberg, Maria~G.},
       title={Relaxation to equilibrium in the one-dimensional
  {C}ahn-{H}illiard equation},
        date={2014},
     journal={SIAM J. Math. Anal.},
      volume={46},
      number={1},
       pages={720\ndash 756},
}

\bib{SW}{article}{
      author={Sun, Xiaodi},
      author={Ward, Michael~J.},
       title={Dynamics and coarsening of interfaces for the viscous
  {C}ahn-{H}illiard equation in one spatial dimension},
        date={2000},
     journal={Stud. Appl. Math.},
      volume={105},
      number={3},
       pages={203\ndash 234},
}

\end{biblist}
\end{bibdiv}


\Addresses

\end{document}